\documentclass[11pt,a4paper,reqno]{amsart}
\usepackage{latexsym}
\usepackage{amsmath}
\usepackage{amssymb}
\usepackage{bbm}
\usepackage{amsthm}
\usepackage{fullpage}

\usepackage{tikz}

\usepackage{fancyhdr}
\usepackage{latexsym}
\usepackage{dsfont}
\usepackage[dvips]{epsfig}
\usepackage[latin1]{inputenc}
\usepackage{appendix}
\usepackage{comment}
\usepackage{color,url}

\usepackage{graphicx}
\usepackage{epsfig}

\usepackage[sc]{mathpazo}
\usepackage[T1]{fontenc}

\newcommand{\bea}{\begin{eqnarray*}}
\newcommand{\eea}{\end{eqnarray*}}
\newcommand{\be}{\begin{eqnarray}}
\newcommand{\ee}{\end{eqnarray}}
\newcommand{\beq}{\begin{equation}}
\newcommand{\eeq}{\end{equation}}
\newcommand\numberthis{\addtocounter{equation}{1}\tag{\theequation}}

\newtheorem{thm}{Theorem}

\newtheorem{lem}[thm]{Lemma}

\newtheorem{prop}[thm]{Proposition}

\theoremstyle{nonumberplain}

\newcommand{\E}[1]{\ensuremath{\mathbb{E} \left[#1 \right]}}
\newcommand{\Prob}[1]{\ensuremath{\mathbb{P} \left(#1 \right)}}

\newcommand{\var}[1]{\ensuremath{\mathrm{var} \left(#1 \right)}}

\newcommand{\I}[1]{\ensuremath{\mathbbm{1}_{ \{ #1 \} }}}
\newcommand{\R}{\ensuremath{\mathbb{R}}}
\newcommand{\Z}{\ensuremath{\mathbb{Z}}}
\newcommand{\N}{\ensuremath{\mathbb{N}}}

\newcommand{\fl}[1]{\ensuremath{\lfloor #1 \rfloor}}
\newcommand{\ce}[1]{\ensuremath{\lceil #1 \rceil}}

\renewcommand{\subset}{\subseteq}
\newcommand{\convdist}{\ensuremath{\stackrel{d}{\longrightarrow}}}
\newcommand{\convprob}{\ensuremath{\stackrel{p}{\rightarrow}}}
\newcommand{\equidist}{\ensuremath{\stackrel{d}{=}}}

\title{The spread of fire on a random multigraph}
\author{Christina Goldschmidt}
\address{Department of Statistics, University of Oxford, 24-29 St Giles', Oxford OX1 3LB, UK}
\email[Christina Goldschmidt]{goldschm@stats.ox.ac.uk}
\author{Eleonora Krea\v{c}i\'{c}}
\email[Eleonora Krea\v{c}i\'{c}]{eleonora.kreacic@stcatz.ox.ac.uk}
\begin{document}

\date{\today}
\keywords{Random multigraph, Karp-Sipser algorithm, differential equations method, reflected stochastic differential equation}
\subjclass{60C05 (Primary); 60F17, 05C80, 05C70, 05C85 (Secondary)}
\maketitle

%
%
\begin{abstract}
We study a model for the destruction of a random network by fire.  Suppose that we are given a multigraph of minimum degree at least 2 having real-valued edge-lengths.  We pick a uniform point from along the length and set it alight; the edges of the multigraph burn at speed 1.  If the fire reaches a vertex of degree 2, the fire gets directly passed on to the neighbouring edge; a vertex of degree at least 3, however, passes the fire either to all of its neighbours or none, each with probability $1/2$.  If the fire goes out before the whole network is burnt, we again set fire to a uniform point.  We are interested in the number of fires which must be set in order to burn the whole network, and the number of points which are burnt from two different directions. We analyse these quantities for a random multigraph having $n$ vertices of degree 3 and $\alpha(n)$ vertices of degree 4, where $\alpha(n)/n \to 0$ as $n \to \infty$, with i.i.d.\ standard exponential edge-lengths. Depending on whether $\alpha(n) \gg \sqrt{n}$ or $\alpha(n)=O(\sqrt{n})$, we prove that as $n \to \infty$ these quantities converge jointly in distribution when suitably rescaled to either a pair of constants or to (complicated) functionals of Brownian motion.

We use our analysis of this model to make progress towards a conjecture of Aronson, Frieze and Pittel concerning the number of vertices which remain unmatched when we use the Karp-Sipser algorithm to find a matching on the Erd\H{o}s-R\'enyi random graph.
\end{abstract}

%
%
\section{Introduction}

\subsection{The model} \label{Sect Model}

Suppose that we are given a finite connected multigraph with strictly positive real-valued edge-lengths.  We introduce a model for the destruction of such a network by fire.  The edges are flammable.  First, a point is picked uniformly (i.e.\ according to the normalised Lebesgue measure) and set alight.  (With probability 1, this point will lie in the interior of an edge.)  The fire passes at speed 1 along the edge (in both directions) until it reaches a node.  A node of degree $d \ge 3$ will pass the fire onto all of its other neighbouring edges with probability $1/2$ or stop the fire with probability $1/2$.  If it stops the fire, it becomes a vertex of degree $d-1$.  A vertex of degree $2$ necessarily passes a fire arriving along from one of its neighbouring edges onto the other one.  The fire spreads until it either goes out or has burnt the whole network.  If it goes out before the whole network is burnt, a new uniform point is picked and set alight, and the process continues as before.

We are interested in two aspects of this process:
\begin{enumerate}
\item How many new fires must be set in order to burn the whole network?
\item How many times does it happen that a point is burnt from two different directions?
\end{enumerate}
We refer to the second phenomenon as a \emph{clash}.

Let $\alpha: \N \to \N$ be a function such that $\alpha(n)/n \to 0$ as $n \to \infty$.  We will study these questions in the setting where the base network is a random multigraph with $n$ vertices of degree $3$ and $\alpha(n)$ vertices of degree 4, sampled according to the configuration model (see below for a description).  Throughout this paper, we will implicitly assume $n$ to be even.   We take the edge-lengths to be independent and identically distributed standard exponential random variables.

Let $F^n$ be the number of fires we must set in order to burn the whole network. 
Let $C^n$ be the number of clashes.  We will study the limiting behaviour of these quantities as $n \to \infty$.  It turns out that both scale as $\sqrt{n}$ as long as $\alpha(n) = O(\sqrt{n})$ and as $\alpha(n)$ if $\alpha(n) \gg \sqrt{n}$.  In order to state our results more precisely, we introduce an auxiliary stochastic process.  

Let $(B_t)_{t \ge 0}$ be a standard Brownian motion and for $a\ge0$ let $(X^a_t, L_t^a)_{0 \le t \le 1}$ be the unique solution to the stochastic differential equation with reflection determined by
\begin{equation}
dX_t^{a} = -\frac{2}{3} \frac{X_t^a}{(1-t)}dt + \frac{2a}{3}(1-t)^{1/3} dt + dB_t + dL^a_t, \quad 0 \le t < 1, \label{eqn:Xdefn}  
\end{equation}
where $(L_t^a)_{0 \le t < 1}$ is the local time process of $X^a$ at level 0.  This solution may be explicitly written as a functional of the Brownian motion as follows.  First define
\[
K_t^a = -\inf_{0 \le s \le t} \left[ a - a(1-s)^{2/3} + \int_0^s \frac{dB_u}{(1-u)^{2/3}} \right], \quad 0 \le t < 1.
\]
Then set
\begin{align*}
L_t^a & = \int_0^t (1-s)^{2/3} dK_s^a, \quad 0 \le t < 1, \\
X_t^a & = a(1-t)^{2/3} - a(1-t)^{4/3} + (1 - t)^{2/3}  \int_0^t \frac{d B_u}{(1-u)^{2/3}} + (1-t)^{2/3} K_t^a,  \quad 0 \le t < 1.
\end{align*}
We show that both of these quantities have finite almost sure limits as $t \to 1$, which we call $L_1^a$ and $X_1^a$.

\begin{thm} \label{thm:main}
\begin{itemize}
\item[(i)] Suppose that $\alpha(n)/\sqrt{n} \to a$ as $n \to \infty$, where $a \ge 0$.  Then, as $n \to \infty$,
\[
\frac{1}{\sqrt{n}} (F^n, C^n) \convdist \left(\frac{1}{2}L^a_1, \int_0^1 \frac{X^a_s}{3(1-s)} ds \right).
\]
where the limiting random variables are almost surely finite.
\item[(ii)] Suppose that $\alpha(n) \gg\sqrt{n}$.  Then, as $n \to \infty$,
\[
\frac{1}{\alpha(n)} F^n \convprob 0, \quad \frac{1}{\alpha(n)} C^n \convprob \frac{1}{4}.
\]
\end{itemize}
\end{thm}

\subsection{Motivation: the Karp-Sipser algorithm on a random graph} \label{Sect KS and fire prop-motivation}

The Karp-Sipser algorithm, introduced in \cite{KarpSipserER}, is a greedy algorithm for finding a matching in a fixed graph. The algorithm works as follows. Now call a vertex of degree 1 a \emph{pendant} vertex.  If there is at least one pendant vertex in the graph, choose one uniformly at random and include the edge incident to it in the matching. Remove this edge, the two vertices that form it and any other edges incident to them. If, on the other hand, there are no pendant vertices in the graph, choose one of the existing edges uniformly at random, and include it in the matching. Remove the chosen edge together with the two vertices that form it, as well as any other edges incident to those vertices. Now repeat the procedure on the resulting graph, until there are no edges remaining.

A key observation is that whenever there exists a pendant vertex in a graph, that vertex and its neighbour are included in \emph{some} maximum matching.  So the Karp-Sipser algorithm never makes a ``mistake'' in including such an edge in its matching.  On the other hand, in the other type of move (picking a uniform edge and including it in the matching) it is possible that it includes an edge which would \emph{not} be in any maximum matching.

The Karp-Sipser algorithm turns out to be very successful at finding a near-maximum matching in certain classes of (sparse) random graphs~\cite{KarpSipserER,KarpSipserRevisited, KarpSipserFixedDegree}.  Suppose that we take the graph to be the Erd\H{o}s--R\'enyi random graph $G(N,M)$, with $N$ vertices and $M$ edges, where $M = \fl{cN/2}$ and $c > 0$ is a constant.  Let $D_N$ be the difference between the size of a maximum matching on $G(N,M)$ and the matching produced by the Karp-Sipser algorithm.  Let $A_N$ be the number of vertices remaining in the graph at the point of the first uniform random choice which are left unmatched by the Karp-Sipser algorithm.  Then $D_N \le A_N/2$.  Aronson, Frieze and Pittel~\cite{KarpSipserRevisited} proved the following theorem.  

\begin{thm}
\begin{enumerate}
\item[(i)] If $c < e$, then $\Prob{D_N = 0} \to 1$ as $N \to \infty$.
\item[(ii)] If $c > e$, there exist constants $C_1, C_2 > 0$ such that
\[
C_1 N^{1/5}/(\log N)^{75/2} \le \E{A_N} \le C_2 N^{1/5} (\log N)^{12}.
\]
\end{enumerate}
\end{thm}

Aronson, Frieze and Pittel conjecture that, in fact, $N^{-1/5} \E{A_N}$ converges as $N \to \infty$; indeed, one might also reasonably conjecture that $N^{-1/5} A_N$ possesses a limit in distribution as $N \to \infty$.  One of our aims in this paper is to make progress towards understanding how such a limit in distribution might arise.

The Karp-Sipser algorithm proceeds in two phases.  In the first phase (Phase I), the algorithm recursively attacks the pendant subtrees in the graph, matching as it goes.  Phase II starts the first time that the algorithm is forced to pick a uniform edge.  At the start of Phase II, the graph necessarily contains only vertices of degree 2 or more.  It will, in general, have several components, of which some may consist of isolated cycles.  In these cycle components, Karp-Sipser necessarily yields a maximum matching.  So the source of the ``mistakes'' is the complex components of the graph present at the start of Phase II.  Our original motivation for introducing the model studied in this paper is to understand the behaviour of the Karp-Sipser algorithm on the type of complex components appearing in Phase II.

For $c < e$, Phase I is essentially the whole story, apart from a few isolated cycles (which cannot cause ``mistakes'').  On the other hand, for $c > e$, a non-trivial graph remains at the end of Phase I.  The work of Aronson, Frieze and Pittel~\cite{KarpSipserRevisited} suggests that, for $c > e$, the part of the Karp-Sipser process which gives the dominant contribution to $A_N$ arises close to the end of Phase II.  Their analysis indicates the following heuristic picture for the structure of the graph at this point: ignoring log-corrections, it is approximately a uniform random graph with $\Theta_p(N^{3/5})$ vertices of degree 2, $\Theta_p(N^{2/5})$ vertices of degree 3 and $\Theta_p(N^{1/5})$ vertices of degree 4.  In a moment, we will describe the structure of this graph.

Before going any further, it will be useful to introduce the configuration model~\cite{BenderCanfield,BollobasConfiguration,WormaldThesis,WormaldConnectivity,WormaldShortCycles}.  Fix $n$ and a sequence $\mathbf{d} = (d_1, d_2, \ldots, d_n)$, where $d_i \ge 0$ is the degree of vertex $i$ and $\sum_{i=1}^nd_i$ is even.  First assign vertex $i$ a number $d_i$ of half-edges. Then choose a uniform random pairing of the half-edges.  This generates a random multigraph $CM(\mathbf{d})$.  For a fixed multigraph $\mathcal{M} = ([n],E(\mathcal{M}))$ having degrees $\mathbf{d}$, $s$ self-loops and $m_e$ the multiplicity of $e \in E(\mathcal{M})$, we have
\[
\Prob{CM(\mathbf{d}) = \mathcal{M}} \propto \frac{1}{2^s \prod_{e \in E(\mathcal{M})} m_e!}.
\]
In particular, assuming that there exists at least one simple graph with degrees $\mathbf{d}$, conditionally on the event that $CM(\mathbf{d})$ is simple, it is uniformly distributed on the set of simple graphs with those degrees.  (See the account in Chapter 7 of van der Hofstad~\cite{RemcoVol1} for proofs of these results.)

For a fixed $k \ge 2$, we will write $\mathrm{Dirichlet}_k(1,1,\ldots,1)$ for the uniform distribution on the simplex $\{\mathbf{x}= (x_1, x_2, \ldots, x_k): x_1, x_2, \ldots, x_k \ge 0, \sum_{i=1}^k x_i = 1\}$ (which is the simplest of the family of Dirichlet distributions). We will omit the subscript $k$ when the number of co-ordinates is clear from context.

Now fix $t, u > 0$ and $v \ge 0$, and suppose we start from a degree sequence with $\fl{tN^{3/5}}$ vertices of degree 2, $\fl{uN^{2/5}}$ of degree 3 and $\fl{vN^{1/5}}$ of degree 4.  Let $G_N$ be a uniform random graph with these degrees.  Let $K_N$ be the \emph{kernel} of $G_N$, that is the multigraph obtained by contracting paths of vertices of degree 2.  It is straightforward to see that $K_N$ is distributed according to the configuration model with $\fl{uN^{2/5}}$ vertices of degree 3 and $\fl{vN^{1/5}}$ vertices of degree 4. With probability tending to 1 as $N \to \infty$, $G_N$ possesses a giant complex component $C_N$, containing all of the vertices of degrees 3 and 4 (i.e.\ the kernel), as well as some random number $t'_N$ of the vertices of degree 2, where $t'_N/N^{3/5} \convprob t$ as $n \to \infty$.  On this event of high probability, $C_N$ is a uniform connected graph with its degree sequence.  Outside the giant, there is a collection of $O_p(\log N)$ disjoint cycles, containing the remaining vertices of degree 2.\footnote{We have not found a good reference for these statements, which are essentially folklore in the random graphs literature; statements in a similar spirit may be found, for example, in the recent paper of Joos, Perarnau, Rautenbach and Reed~\cite{Joosetal}. In particular, their Theorem 2 implies that $G_N$ contains a giant component with high probability. By Theorem 4.14 of van der Hofstad~\cite{RemcoVol2}, $K_N$ is connected with high probability.  That the giant component of $G_N$ contains $K_N$ and a proportion 1 of all vertices of degree 2 essentially comes down to the fact that the number of ways of generating a random 2-regular graph is negligible compared to the number of ways of generating a graph with kernel $K_N$.  (We omit the details since our primary aim here is not to make rigorous statements but rather to give heuristics.) Finally, a random 2-regular graph on $n$ vertices has $O_p(\log n)$ components (see, for example, Arratia, Barbour and Tavar\'e~\cite{ArratiaBarbourTavare}).}

Now consider the following way of constructing a complex component which we claim has approximately the same distribution as $C_N$.  First generate the kernel $K_N$ according to the configuration model with $\fl{uN^{2/5}}$ vertices of degree 3 and $\fl{vN^{1/5}}$ vertices of degree 4.  Then, one-by-one, allocate $t_N'$ vertices of degree 2 to the edges of $K_N$: at each step, an edge of the current structure is chosen uniformly at random, split into two edges in series, and the vertex is inserted into the middle.  Thus the lengths of the paths of degree-2 vertices we insert between neighbouring vertices in $K_N$ evolve according to a multicolour P\'olya's urn with $(3\fl{uN^{2/5}} + 4 \fl{vN^{1/5}})/2$ colours and a single ball of each colour to start.\footnote{By Proposition 7.13 of \cite{RemcoVol1}, $K_N$ possesses constant-order numbers of self-loops and multiple edges (indeed, the numbers of self-loops and multiple edges converge jointly in distribution as $N \to \infty$ to a pair of independent Poisson(1) random variables); in the urn process, there is negligible probability that we fail to allocate any vertices of degree 2 to the self-loops or to at least two of a set of edges between the same two vertices.}  Then for large $N$, the proportions of the $t_N'$ vertices of degree 2 which get allocated to each of the edges of $K_N$ look approximately like a $\mathrm{Dirichlet}(1,1,\ldots,1)$ vector, with $(3\fl{uN^{2/5}} + 4 \fl{vN^{1/5}})/2$ co-ordinates. 

Let us now consider how the Karp-Sipser algorithm behaves on $G_N$.  We may neglect the cycle components, as they can give rise to at most $O(\log N)$ unmatched vertices.  The algorithm first picks a uniform edge, matches its end-points, and then removes the neighbouring edges.  Typically we matched two vertices somewhere inside a long path of degree 2 vertices, and so we are now left with two pendant vertices, each at the end of a path of degree-2 vertices with a higher-degree vertex at its end.  If such a path is of odd length, Karp-Sipser will ``consume'' the degree-2 vertices but leave the higher-degree vertex untouched (except to reduce its degree by 1).  If, on the other hand, the path is of even length, the higher-degree vertex gets matched and removed, causing its neighbours to become pendant vertices.  Thus, in this case, the algorithm eats further away into the graph.  If ever a particular path gets eaten away at from both ends (which can happen since the graph has cycles), there is a chance that some vertex in the path will remain unmatched.  Again, whether this in fact happens or not depends on the parity of the path of degree-2 vertices.  For large enough $N$, we expect that such paths will be of odd and even lengths with approximately equal probability.  So we expect that paths which get burnt from both ends will give rise to an unmatched vertex with probability $1/2$.  

As we have already argued, since the number of edges in the multigraph is much smaller than the number of degree-2 vertices, the proportions of degree-2 vertices assigned to each edge of the multigraph will be, for large $N$, close to $\mathrm{Dirichlet}(1,1,\ldots,1)$.  For fixed $k \ge 2$, a vector $(D_1, D_2, \ldots, D_k)$ with $\mathrm{Dirichlet}_k(1,1,\ldots,1)$ distribution may be obtained by sampling $E_1, E_2, \ldots, E_k$, independent and identically distributed standard exponential random variables and setting
\begin{equation} \label{eqn:dir}
(D_1, D_2, \ldots, D_k) = \frac{1}{\sum_{i=1}^k E_i} (E_1, E_2, \ldots, E_k).
\end{equation}
The right-hand side is independent of the random variable $\sum_{i=1}^k E_i$.  Once we have accounted for parity, the lengths of the paths of degree-2 vertices play a role only when we pick a new uniform edge to match, which we do with probability proportional to length. So only \emph{relative} lengths matter, and we can equivalently think of $E_1, \ldots, E_k$ as the ``lengths'' of the paths of degree 2 vertices in our approximate model.  (In what follows, we will primarily use $\mathrm{Dirichlet}(1,1,\ldots,1)$ edge-lengths, but it is convenient to be able to move back and forth between these two points of view.)

In summary, letting $n$ be the number of vertices of degree 3 and $\alpha(n)$ be the number of vertices of degree 4, we obtain the model described in Section~\ref{Sect Model} as an approximation.  We do not attempt here to assess the quality of this approximation (and we only make rigorous statements about the model described in Section~\ref{Sect Model}).  Rather our interest is in the mechanism by which the distribution of the number of vertices which remain unmatched at the end of the Karp-Sipser algorithm arises.  
With the scaling suggested by Aronson, Frieze and Pittel, we would have $n = \Theta(N^{2/5})$ and $\alpha(n) = \Theta(N^{1/5})$ i.e.\ $\alpha(n) = \Theta(\sqrt{n})$. So we should be in regime (i) of Theorem~\ref{thm:main}, which supports the conjecture that $N^{-1/5} A_N$ possesses a limit in distribution.

\subsection{Our analysis}

Our model has two convenient features which make it amenable to analysis: the distributional properties of the edge-lengths and the fact that we may sample the multigraph edge by edge at the same time that we burn it.  Let us first address the edge-lengths.  We will make use of the following result (Proposition 1 of \cite{BertoinG}), which follows from standard properties of exponential random variables via the relationship (\ref{eqn:dir}).

\begin{prop} \label{prop:exponentialmagic}
(i) Suppose that $(D_1, D_2, \ldots, D_k) \sim \mathrm{Dirichlet}_k(1,1,\ldots,1)$ and let $I$ be a random index from $\{1,2,\ldots, k\}$ with conditional distribution
\[
\Prob{I = i| D_1, \ldots, D_k} = D_i, \quad 1 \le i \le k.
\]
Let $U$ be independent of everything else with uniform distribution on $[0,1]$.  Then defining
\[
D'_i = \begin{cases}
D_i & \text{ for $1 \le i < I$}, \\
U D_i & \text{ for $i = I$}, \\
(1-U) D_i & \text{ for $i = I+1$}, \\
D_{i-1} & \text{ for $I+1 < i \le k+1$},
\end{cases}
\]
we have that $(D'_1, D'_2, \ldots, D'_{k+1})$ has $\mathrm{Dirichlet}_{k+1}(1,1,\ldots,1)$ distribution.

(ii) Suppose that $(D'_1, D'_2, \ldots, D'_{k+1}) \sim \mathrm{Dirichlet}_{k+1}$ and that $J$ is chosen independently and uniformly at random from $\{1,2,\ldots, k\}$. Then defining
\[
D_i = \begin{cases}
D'_i & \text{ for $1 \le i < J$},\\
D'_i+D'_{i+1} & \text{ for $i = J$},\\
D'_{i+1} & \text{ for $J+1 < i \le k$},
\end{cases}
\]
we have that $(D_1, D_2, \ldots, D_k) \sim \mathrm{Dirichlet}_k(1,1,\ldots,1)$.
\end{prop}

Using part (i) of this proposition, we see that if we start from $\mathrm{Dirichlet}(1,1,\ldots,1)$ edge-lengths and pick a point uniformly from the length measure, then splitting at it yields distances from the sampled point to the adjacent vertices of degree at least 3 either side of it which are again part of a $\mathrm{Dirichlet}(1,1,\ldots,1)$ vector (with one more co-ordinate).  We will see in what follows that the property that (given the number of edges in the multigraph) the relative edge-lengths are $\mathrm{Dirichlet}(1,1,\ldots,1)$ is preserved. Consider the process at an instant when the fire has just reached a vertex, or when a point has just been set alight.  In general there are several edges burning, and using the memoryless property of the exponential distribution, their relative distances to the adjacent vertices are still standard exponential divided by the sum of those exponentials.  In particular, whenever there are multiple edges burning, the next to reach its vertex is chosen uniformly from among all those present.  Since we are not interested in how long the whole process takes but rather in the number of times we must set light to the network and how many clashes we observe, we perform our analysis in discrete time: that is, we consider the multigraph \emph{without} edge-lengths, always set fire to a uniformly-chosen edge (which has the effect of splitting it into two edges, both alight), and we choose which edge next finishes burning uniformly from among those currently alight.

Recall the description of the configuration model from the previous section.   We may generate the pairing of the half-edges in any order we like, which makes the configuration model particularly amenable to an exploration-process-type analysis.  In particular, given that we have revealed the pairings of a particular collection of half-edges, assuming we keep track of the degree sequence of the rest of the graph, the rest of the graph is again a configuration model with that degree sequence.  We exploit this property below.  Observe that we have $n$ vertices of degree 3 and $\alpha(n)$ of degree 4 in our configuration model.  Our process starts by picking a uniform edge, splitting it in two and setting each resulting edge alight.  (Let us refer to this as the beginning of a \emph{wave}, with a new wave beginning every time we set light to a point in the multigraph.)  A uniform one of these two edges reaches its vertex next.  It samples its vertex from among those of degree 3 with probability $3n/(3n+4\alpha(n))$ and from those of degree 4 with probability $4\alpha(n)/(3n+4\alpha(n))$.  Thereafter, we may think of edges which are alight as the first half-edge of a pair, whose second half-edge we have yet to sample.  

Suppose that we currently have $x$ edges burning, and that there are $u$ vertices with 3 unattached half-edges and $v$ vertices with 4 unattached half-edges.  We pick the next half-edge to process uniformly from those currently burning, and pick its pair uniformly at random from among those available, including any which are themselves burning.  The pair half-edge is already burning with probability $(x-1)/(3u + 4v + x-1)$, in which case we form an edge which is burning from both ends and  generate a clash.  Otherwise, if we connect to a vertex of degree 3, which occurs with probability $3u/(3u+4v+x-1)$, the fire is either passed to the two other half-edges (with probability $1/2$) or stopped.  If it is stopped, the vertex of degree 3 becomes a vertex of degree 2.  There are two different things that might happen to this vertex of degree 2.  With probability $1/(3u +4v +x -3)$, its two half-edges are in fact connected to each other to form an isolated cycle. (As observed above, this is a rare event: there are only $O(1)$ self-loops in the whole multigraph.)  Any such isolated cycle necessarily yields a clash. With the complementary probability, the two remaining half-edges are not connected to each other but rather to other half-edges. Since vertices of degree 2 cannot stop fires, in this case we may simply remove the vertex and contract the path of length 2 in which it sat to a single edge.  (Using part (ii) of Proposition~\ref{prop:exponentialmagic}, this results the relative lengths of the edges in the unseen parts of the multigraph still being $\mathrm{Dirichlet}(1,1,\ldots,1)$ distributed).  In either case, the remaining unseen part of the multigraph (after deletion of the loop, or contraction of a path of length 2) is still distributed according to the configuration model with the updated degree distribution.  Finally, if we connect to a vertex of degree 4, which occurs with probability $4v/(3u+4v+x-1)$, then either the fire is passed to all three other neighbours (with probability $1/2$) or to none of them, in which case the result is that we get another vertex of degree 3.  

Note that we must treat the very first edge of a wave differently: although we start with two burning half-edges, they cannot be paired to each other (since otherwise there would be an edge in the original graph with no vertex, which is impossible).  So let us treat the first step of a wave as consisting of picking a uniform edge, splitting it in two, sampling the vertex to which one of the resulting burning edges is attached and seeing whether it passes the fire on or not.  So if at some step, there are no burning half-edges, on the next step there will be either 1, 3 or 4 burning half-edges, corresponding to the events that the first of the two fires was stopped, that it was passed on through a vertex of degree 3, or that it was passed on through a vertex of degree 4.

In this way, we see that at each step of the procedure, we either process one vertex or generate a clash.  A vertex of degree 3 is processed precisely once; a vertex of degree 4 is processed once or twice, depending on whether it stops the first fire it encounters or not.

If we track the numbers of vertices of degree 3 and 4 and the numbers of currently burning half-edges, we have a Markovian evolution, which we may hope to analyse using the tools of stochastic process theory.  In particular, in what follows we make extensive use of martingales.

The rest of the paper is organized as follows. In Section~\ref{Sect Markov chain}, we write down explicitly the transition probabilities of our Markov chain, and give some first estimates relevant for the forthcoming analysis. In particular, we identify a coupling which facilitates our analysis of the end of the process. In Section~\ref{Sect fluid limit approximations}, we prove fluid limits for the suitably rescaled number of nodes of degree $3$ and $4$; that is, we show that these processes remain close to deterministic functions on a time-interval which is bounded away from the end of the process (this is an application of the so-called differential equations method \cite{DarlingNorris, Wormald}). In Section~\ref{Sect Fluid limit X N large omega}, we analyse the case $\alpha(n)\gg\sqrt{n}$, and prove a fluid limit result for the (suitably rescaled) numbers of fires and clashes we observe, as long as we are bounded away from the end of the process. Section~\ref{Sect Diffusion limit small omega} deals with the limiting properties of the numbers of fires and the number of clashes we observe, again as long as we bounded away from the end of the process, for $\alpha(n)=O(\sqrt{n})$. In this case, the limiting process for the number of fires is a reflected diffusion, and the proof is based on an invariance principle for reflecting Markov chains (in the spirit of \cite{EthierKurtz, KangWilliams, StroockVaradhan}). Finally, in Section~\ref{sec:end}, we prove that the end of the process does not contribute significantly to any of these quantities, and so the convergence results can be extended into that range also.

%
%
\section{The Markov chain} \label{Sect Markov chain}
For $i \ge 0$, let $U^{n}({i})$ and $V^{n}({i})$ represent the numbers of nodes of degree $3$ and $4$ respectively after $i$ steps of the burning procedure. Let $X^{n}({i})$ be the number of burning half-edges we have after $i$ steps. Let $N^{n}({i})$ be the counting process of the number of clashes observed up to step $i$. 
Set $U^n(0) = n$, $V^n(0) = \alpha(n)$,  $X^n(0) = 0$, and $N^n(0) = 0$.  We have already argued that the process
\[
(U^n(i), V^n(i), X^n(i), N^n(i))_{i \ge 0}
\]
evolves in a Markovian manner until time $\zeta_n = \inf\{i \ge 0: U^n(i) + V^n(i) + X^n(i) = 0\}$, when it stops.  Write
\[
L^n(k) = 2 \sum_{i=0}^{k-1} \I{X^n(i) = 0}
\]
(we will think of this quantity as a local time). Then
\[
F^n = \frac{1}{2} L^n(\zeta_n), \quad C^n = N^n(\zeta_n).
\]
We will find it convenient to rescale time by $n$ (essentially because we start with $n + \alpha(n)$ vertices and a typical step involves the removal of a vertex of degree 3).

Recall the definition of the process $X^a$ from (\ref{eqn:Xdefn}).
We also let $x: [0,1] \to \R_+$ be the (deterministic) function
\[
x(t) = (1-t)^{2/3} - (1-t)^{4/3}, \quad 0 \le t \le 1.
\]
For $0 \le t \le 1$, let
\[
N^a_t = \int_0^t \frac{X^a_s}{3(1-s)} ds, \qquad m(t) = \frac{1}{4}-\frac{1}{2}(1-t)^{2/3}+\frac{1}{4}(1-t)^{4/3}.
\]

Let $\mathbb{D}(\R_+,\R)$ denote the space of c\`adl\`ag functions from $\R_+$ to $\R$, equipped with the Skorokhod topology. We will study the convergence of the sequence of probability measures on the measurable space given by $\mathbb{D}(\R_+,\R)$ endowed with its Borel $\sigma$-algebra. (In fact, since our limit processes will always be continuous, we will rather obtain convergence with respect to the uniform norm.) The crux of our argument is the following scaling limit theorem.

\begin{thm} \label{thm:scalinglimit}
(i) Suppose $\alpha(n)/\sqrt{n} \to a$ as $n \to \infty$, for $a \ge 0$.  Then
\[
\frac{1}{\sqrt{n}} \left(X^{n}({\fl{nt}}), L^n(\fl{nt}), N^n(\fl{nt}), 0 \le t \le \zeta_n/n \right) \convdist (X^a_t, L^a_t, N^a_t, 0 \le t \le 1)
\]
uniformly as $n \to \infty$.

(ii) Suppose $\alpha(n)\gg\sqrt{n}$.  Then
\[
\frac{1}{\alpha(n)}  \left(X^{n}(\fl{nt}), L^n(\fl{nt}), N^n(\fl{nt}), 0 \le t \le \zeta_n/n \right) \convdist (x(t), 0, m(t) , 0 \le t \le 1)
\]
uniformly as $n \to \infty$.
\end{thm}


It is straightforward to see that this implies Theorem~\ref{thm:main}.

\subsection{Transition probabilities}

We have already described the possible transitions of our four-dimensional Markov chain in the Introduction.  Let us be a little more explicit about the transition probabilities.

Suppose that $(U^n(i), V^n(i), X^n(i), N^n(i)) = (u,v,x,m) \in \Z_+^4$.  If $x = 0$ but $u+v > 0$, then
\begin{align*}
& (U^n(i+1), V^n(i+1), X^n(i+1), N^n(i+1))  \\
& \qquad = \begin{cases}
(u-1, v, 1, m+1) & \text{with probability $\frac{3u}{2(3u + 4v)(3u+4v-2)}$,} \\
(u-1, v, 1, m) & \text{with probability $\frac{3u(3u+4v-3)}{2(3u + 4v)(3u+4v-2)}$,} \\
(u-1, v, 3, m) & \text{with probability $\frac{3u}{2(3u+4v)}$,} \\
(u+1, v-1, 1, m) & \text{with probability $\frac{4v}{2(3u+4v)}$,} \\
(u, v-1, 4, m) & \text{with probability $\frac{4v}{2(3u+4v)}$}. \\
\end{cases}
\intertext{If $x > 0$ then}
& (U^n(i+1), V^n(i+1), X^n(i+1), N^n(i+1)) \\
& \qquad = 
\begin{cases}
(u-1, v, x-1, m+1) & \text{with probability $\frac{3u}{2(3u+4v + x-1)(3u + 4v + x-3)}$,} \\
(u-1, v, x-1, m) & \text{with probability $\frac{3u(3u+4v+x-4)}{2(3u+4v+x-1)(3u+4v+x-3)}$,} \\
(u-1, v, x+1, m) & \text{with probability $\frac{3u}{2(3u+4v+x-1)}$,} \\
(u+1, v-1, x-1, m) & \text{with probability $\frac{4v}{2(3u+4v+x-1)}$,} \\
(u, v-1, x+2, m) & \text{with probability $\frac{4v}{2(3u+4v+x-1)}$,} \\
(u, v, x-2, m+1) & \text{with probability $\frac{x-1}{(3u+4v+x-1)}$}. \\
\end{cases}
\end{align*}


Writing $(\mathcal{F}_i^n)_{i \ge 0}$ for the natural filtration of the four-dimensional process, we have
\begin{align*}
\E{U^n(i+1) - U^n(i) | \mathcal{F}^n_i} 
& = \frac{4V^n(i) - 3 U^n(i)}{3 U^n(i) + 4 V^n(i) + X^n(i) - \I{X^n(i) > 0}},  \\
\E{V^n(i+1) - V^n(i) | \mathcal{F}^n_i}
& = \frac{-4V^n(i)}{3 U^n(i) + 4 V^n(i) + X^n(i) - \I{X^n(i) > 0}}, \\
\E{X^n(i+1) - X^n(i) |  \mathcal{F}^n_i} 
& = 2 \I{X^n(i) = 0} + \frac{2V^n(i) -2X^n(i) + 2 \I{X^n(i) > 0}}{3U^n(i) + 4 V^n(i) + X^n(i) -  \I{X^n(i) > 0}},  \numberthis \label{eq exp changes X} \\
 \E{N^n(i+1) - N^n(i) |  \mathcal{F}^n_i} 
& = \tfrac{3U^n(i) + 2(X^n(i) - \I{X^n(i) > 0})(3U^n(i) + 4V^n(i) + X^n(i) - 3)}{2(3U^n(i) + 4V^n(i) + X^n(i) -\I{X^n(i) > 0})(3U^n(i) + 4V^n(i) + X^n(i) -\I{X^n(i) > 0} - 2)}. 
\end{align*}
We will also need the following conditional second moments:
\begin{align}
 \E{(U^n(i+1) - U^n(i))^2 | \mathcal{F}^n_i} & = \frac{3U^n(i) + 2 V^n(i)}{3U^n(i) + 4V^n(i) + X^n(i) - \I{X^n(i) > 0}}, \notag \\ 
 \E{(V^n(i+1) - V^n(i))^2 | \mathcal{F}^n_i} & = \frac{4V^n(i)}{3U^n(i) + 4V^n(i) + X^n(i) - \I{X^n(i) > 0}}, \notag \\ 
 \E{(X^n(i+1) - X^n(i))^2 | \mathcal{F}^n_i} &
= \begin{cases}
5 + \frac{14V^n(i)}{3U^n(i) + 4V^n(i)} & \text{if $X^n(i) = 0$,}\\
1 + \frac{6V^n(i) + 3X^n(i) -3}{3U^n(i) + 4V^n(i) + X^n(i) - 1} & \text{if $X^n(i) > 0$},\\
\end{cases} \label{eqn:squareincrX} \\
\E{(N^n(i+1) - N^n(i))^{2} |  \mathcal{F}^n_i} 
& = \E{N^n(i+1) - N^n(i) | \mathcal{F}^n_i}. \notag
\end{align}

\subsection{A coupling and some first estimates} \label{sec:coupling}
The process runs for $\zeta_n$ steps.  On each step, we process a whole edge of the multigraph, except at the start of a wave, when we possibly need two steps to process an edge.  There are $(3n+4\alpha(n))/2$ edges in total and so we get the crude bound
\begin{equation} \label{eqn:crudebound}
\zeta_n \le 3n + 4\alpha(n).
\end{equation}

In the sequel, and particularly in Section~\ref{sec:end}, we will make extensive use of a coupling of a modified version of $X^n$ and a reflecting simple symmetric random walk, which we now introduce.  First, we divide the burning half-edges into two stacks, of sizes $X^n_1(i)$ and $X^n_2(i)$, such that $X^n(i) = X^n_1(i) + X^n_2(i)$ for all $0 \le i \le \zeta_n$.  We may give these sub-processes whatever dynamics we choose, as long as their sum behaves as $(X^n(i))_{i \ge 0}$.  We proceed as follows. 

Whenever $X^n_1(i) > 0$, we select the next half-edge to process from the first stack.  If the fire is absorbed, $X^n_1$ is simply reduced by 1.  If we connect to a vertex of degree 3 and pass the fire on, $X^n_1$ increases by 1.  If we connect to a vertex of degree 4 and pass the fire on, we let each of $X^n_1$ and $X^n_2$ increase by 1.  If we create a clash with a half-edge from the first stack, $X^n_1$ is reduced by 2.  We may also create a clash with a vertex from the second stack, in which case $X^n_1$ and $X^n_2$ are both reduced by 1.

If $X^n_1(i) = 0$ but $X^n_2(i) > 0$, then we select the half-edge to process from the second stack, add any new half-edges arising from passing the fire on to a vertex of degree 3 to $X_1^n$ and add 2 of the three burning half-edges arising from passing the fire on to a vertex of degree 4 to $X_1^n$ and the last one to $X^n_2$.  Finally, if $X^n_1(i) = X^n_2(i) = 0$, then we allocate all new half-edges to the first stack, except if we connect to a vertex of degree 4 and pass on the fire, in which case $X^n_1$ jumps to 3 and $X^n_2$ jumps to 1.

We will also track the clashes and split them according to whether they involve a half-edge from the second stack or not, yielding $N^n(i) = N_1^n(i) + N_2^n(i)$ for $i \ge 0$.

We will describe the transition probabilities of this process in detail below.  Our aim is to couple $X_1^n$ with a process $Y^n$ in such a way that $X_1^n(i) \le Y^n(i)$ for all $i \ge 0$ and $Y^n$ is a simple symmetric random walk (SSRW), reflected at 2.  In order to keep the notation to a reasonable level, we will describe the transitions of $(X_1^n(i), X_2^n(i), Y^n(i), N^n_1(i), N_2^n(i))_{i \ge 0}$, leaving those of $(U^n(i), V^n(i))_{i \ge 0}$ implicit.

Conditional on $X_1^n(i) = x_1 > 0, X_2^n(i) = x_2 \ge 0, Y^n(i) = y \ge 3, U^n(i) = u, V^n(i) = v$, with $x=x_1+x_2$,
\[
\begin{aligned}
& (X_1^n(i+1) - X_1^n(i), X_2^n(i+1) - X_2(i), Y^n(i+1) - Y^n(i), N^n_1(i+1) -N^n_1(i),  N^n_2(i+1) -N^n_2(i)) \\
& \qquad \qquad = \begin{cases} 
(-2, 0, -1, +1, 0) & \text{with probability $\frac{x_1-1}{2(3u+4v+x-1)}$} \\
(-2,0, +1, +1, 0) & \text{with probability $\frac{x_1-1}{2(3u+4v+x-1)}$} \\
(-1, -1, -1, 0, +1) & \text{with probability $\frac{x_2}{2(3u+4v+x-1)}$} \\
(-1,-1, +1, 0, +1) & \text{with probability $\frac{x_2}{2(3u+4v+x-1)}$} \\
(-1, 0, -1, +1, 0) & \text{with probability $\frac{3u}{2(3u+4v + x-1)(3u + 4v + x-3)}$} \\
(-1,0, -1, 0, 0) & \text{with probability $\frac{4v}{2(3u+4v+x-1)} + \frac{3u(3u+4v+x-4)}{2(3u+4v + x-1)(3u + 4v + x-3)}$} \\
(+1,0, +1, 0, 0) & \text{with probability $\frac{3u}{2(3u+4v+x-1)}$} \\
(+1,+1,+1, 0, 0) & \text{with probability $\frac{4v}{2(3u+4v+x-1)}$}. \\
\end{cases}
\end{aligned}
\]
Conditional on $X_1^n(i) = x_1 > 0, X_2^n(i) = x_2 \ge 0, Y^n(i) = y = 2, U^n(i) = u, V^n(i) = v$, with $x=x_1+x_2$,
\[
\begin{aligned}
& (X_1^n(i+1) - X_1^n(i), X_2^n(i+1) - X_2(i), Y^n(i+1) - Y^n(i), N^n_1(i+1) -N^n_1(i),  N^n_2(i+1) -N^n_2(i)) \\
& \qquad \qquad = \begin{cases} 
(-2, 0, +1, +1, 0) & \text{with probability $\frac{x_1-1}{(3u+4v+x-1)}$} \\
(-1, -1, +1, 0, +1) & \text{with probability $\frac{x_2}{(3u+4v+x-1)}$} \\
(-1, 0, +1, +1, 0) & \text{with probability $\frac{3u}{2(3u+4v + x-1)(3u + 4v + x-3)}$} \\
(-1,0, +1, 0, 0) & \text{with probability $\frac{4v}{2(3u+4v+x-1)} + \frac{3u(3u+4v+x-4)}{2(3u+4v + x-1)(3u + 4v + x-3)}$} \\
(+1,0, +1, 0, 0) & \text{with probability $\frac{3u}{2(3u+4v+x-1)}$} \\
(+1,+1,+1, 0, 0) & \text{with probability $\frac{4v}{2(3u+4v+x-1)}$}. \\
\end{cases}
\end{aligned}
\]
Conditional on $X_1^n(i) = 0, X_2^n(i) = x_2 > 0, Y^n(i) = y \ge 3, U^n(i) = u, V^n(i) = v$,
\[
\begin{aligned}
& (X_1^n(i+1) - X_1^n(i), X_2^n(i+1) - X_2(i), Y^n(i+1) - Y^n(i), N^n_1(i+1) -N^n_1(i),  N^n_2(i+1) -N^n_2(i)) \\
& \qquad \qquad = \begin{cases} 
(0, -2, -1, 0, +1) & \text{with probability $\frac{x_2-1}{2(3u+4v+x_2-1)}$} \\
(0,-2, +1, 0, +1) & \text{with probability $\frac{x_2-1}{2(3u+4v+x_2-1)}$} \\
(0, -1, -1, +1, 0) & \text{with probability $\frac{3u}{2(3u+4v + x_2-1)(3u + 4v + x_2-3)}$} \\
(0,-1, -1, 0, 0) & \text{with probability $\frac{4v}{2(3u+4v+x_2-1)} + \frac{3u(3u+4v+x_2-4)}{2(3u+4v + x_2-1)(3u + 4v + x_2-3)}$} \\
(+2,-1, +1, 0, 0) & \text{with probability $\frac{3u}{2(3u+4v+x_2-1)}$} \\
(+2,0,+1,0, 0) & \text{with probability $\frac{4v}{2(3u+4v+x_2-1)}$}. 
\end{cases}
\end{aligned}
\]
Conditional on $X_1^n(i) =  0, X_2^n(i) = x_2 > 0, Y^n(i) = y =2, U^n(i) = u, V^n(i) = v$,
\[
\begin{aligned}
& (X_1^n(i+1) - X_1^n(i), \ X_2^n(i+1) - X_2(i), \ Y^n(i+1) - Y^n(i), N^n_1(i+1) -N^n_1(i),  N^n_2(i+1) -N^n_2(i)) \\
& \qquad \qquad = \begin{cases} 
(0,-2, +1, 0, +1) & \text{with probability $\frac{x_2-1}{(3u+4v+x_2-1)}$} \\
(0, -1, +1, +1, 0) & \text{with probability $\frac{3u}{2(3u+4v + x_2-1)(3u + 4v + x_2-3)}$} \\
(0,-1, +1, 0, 0) & \text{with probability $\frac{4v}{2(3u+4v+x_2-1)} + \frac{3u(3u+4v+x_2-4)}{2(3u+4v + x_2-1)(3u + 4v + x_2-3)}$} \\
(+2,-1, +1, 0, 0) & \text{with probability $\frac{3u}{2(3u+4v+x_2-1)}$} \\
(+2,0,+1,0, 0) & \text{with probability $\frac{4v}{2(3u+4v+x_2-1)}$}. 
\end{cases}
\end{aligned}
\]
Conditional on $X_1^n(i) =  0, X_2^n(i) = 0, Y^n(i) = y \ge 3, U^n(i) = u, V^n(i) = v$,
\[
\begin{aligned}
& (X_1^n(i+1) - X_1^n(i), X_2^n(i+1) - X_2(i), Y^n(i+1) - Y^n(i), N^n_1(i+1) -N^n_1(i),  N^n_2(i+1) -N^n_2(i)) \\
& \qquad \qquad = \begin{cases} 
(+1,0, -1,0,0) & \text{with probability $\frac{1}{2}$} \\
(+3,0, +1,0,0) & \text{with probability $\frac{3u}{2(3u+4v)}$} \\
(+3,+1,+1,0,0) & \text{with probability $\frac{4v}{2(3u+4v)}$}. \\
\end{cases}
\end{aligned}
\]
Finally, conditional on $X_1^n(i) =  0, X_2^n(i) = 0, Y^n(i) = y = 2, U^n(i) = u, V^n(i) = v$,
\[
\begin{aligned}
& (X_1^n(i+1) - X_1^n(i), X_2^n(i+1) - X_2(i), Y^n(i+1) - Y^n(i), N^n_1(i+1) -N^n_1(i),  N^n_2(i+1) -N^n_2(i)) \\
& \qquad \qquad = \begin{cases} 
(+1,0, +1,0,0) & \text{with probability $\frac{1}{2}$} \\
(+3,0, +1,0,0) & \text{with probability $\frac{3u}{2(3u+4v)}$} \\
(+3,+1,+1,0,0) & \text{with probability $\frac{4v}{2(3u+4v)}$}. \\
\end{cases}
\end{aligned}
\]
By construction, $Y^n$ performs a  SSRW with upward reflection at 2.  For fixed $i \ge 0$, as long as $Y^n(i) \ge 2$ and $X_1^n(i) \le Y^n(i)$, we obtain $X_1^n(j) \le Y^n(j)$ for all $j \ge i$.

Finally, we observe that since a vertex of degree 4 contributes to the size of the second stack at most once, $X^n_2(i) \le V^n(0) - V^n(i)$ for all $i \ge 0$. Similarly, the number of clashes involving at least one half-edge from the second stack is bounded above by the number of vertices of degree 4 processed so far, $N_2^n(i) \le V^n(0) - V^n(i)$.

Henceforth, we will use the notation $(\mathcal{F}_i^n)_{i \ge 0}$ for the natural filtration of the process
\[
(X^n_1(i), X^n_2(i), Y^n(i), U^n(i), V^n(i), N^n_1(i), N^n_2(i))_{i \ge 0}.
\]
For future reference, we note that
\begin{equation} \label{eqn:X_1}
\begin{aligned}
\E{X_1^n(i+1) - X_1^n(i) |  \mathcal{F}^n_i} 
& = \I{X^n_1(i) = 0, X^n_2(i) > 0} + 2 \I{X_1^n(i) = 0, X^n_2(i) = 0} \\
& \qquad \qquad  - \frac{2(X_1^n(i) - \I{X_1^n(i) > 0}) + (X_2^n(i) - \I{X_1^n(i) = 0, X_2^n(i) > 0})}{3U^n(i) + 4 V^n(i) + X^n(i) -  \I{X^n(i) > 0}}
\end{aligned}
\end{equation}
and
\begin{equation} \label{eqn:N_1}
\begin{aligned}
 \E{N_1^n(i+1) - N_1^n(i) |  \mathcal{F}^n_i} 
& = \tfrac{3U^n(i) + 2(X_1^n(i) - \I{X_1^n(i) > 0})(3U^n(i) + 4V^n(i) + X^n(i) - 3)}{2(3U^n(i) + 4V^n(i) + X^n(i) -\I{X^n(i) > 0})(3U^n(i) + 4V^n(i) + X^n(i) -\I{X^n(i) > 0} - 2)}.
\end{aligned}
\end{equation}

As a first consequence of our coupling, we show that $X^n$ varies on a smaller scale than $n$.

\begin{lem} \label{Lem X is little o of n}
We have that $\frac{1}{\alpha(n) \vee \sqrt{n}} \sup_{0 \le i \le \zeta_n} X^{n}(i)$ is bounded in $L^2$. In particular,
\begin{align*}
\frac{1}{n} \sup_{0 \le i \le \zeta_n} X^{n}(i) \convprob 0,
\end{align*}
as $n \to \infty$.
\end{lem}

\begin{proof}
We have $X^n(0) = 0$; let $Y^n(0) = 2$.  Then by the coupling, we have 
\[
\sup_{0 \le i \le \zeta_n} X^n(i) \le  \alpha(n) + \sup_{0 \le i \le \zeta_n} Y^n(i) \le \alpha(n) + \sup_{0 \le i \le 3n+4\alpha(n)} Y^n(i).
\]
Now let $Z$ be a standard SSRW, and note that $Z$ is a martingale.  Then $Y^n$ has the same law as $(2+ |Z(i)|)_{i \ge 0}$.  By Doob's $L^2$ inequality,
\[
\E{\left(\sup_{0 \le i \le 3n+4\alpha(n)} |Z(i)| \right)^2} \le 4 \E{Z(3n+4\alpha(n))^2} = 4(3n + 4\alpha(n)).
\]
Hence,
\begin{align*}
\E{ \left(\sup_{0 \le i \le \zeta_n} X^n(i) \right)^2} & \le \E{\left(\alpha(n) + 2 + \sup_{0 \le i \le 3n + 4\alpha(n)} |Z(i)| \right)^2} \\
&  \le (\alpha(n) + 2)^2 + 4(3n + 4 \alpha(n)) + 4 (\alpha(n) + 2) \sqrt{3n + 4\alpha(n)}.
\end{align*}
It follows that there exists a constant $C > 0$ such that for all $n \ge 1$ we have
\[
\E{\left( \frac{1}{\alpha(n) \vee \sqrt{n}} \sup_{0 \le i \le \zeta_n} X^n(i) \right)^2} \le C.
\]
Since $\alpha(n)/n \to 0$, we have moreover that
\[
\E{\left( \frac{1}{n} \sup_{0 \le i \le \zeta_n}X^n(i) \right)^2} \to 0,
\]
and so $\frac{1}{n} \sup_{0 \le i \le \zeta_n}X^n(i) \convprob 0$.
\end{proof}

%
%
\section{Fluid limit approximations for the auxiliary processes} \label{Sect fluid limit approximations}

In this section we prove that $U^{n}$ and $V^{n}$, after appropriate rescaling, remain concentrated around their expected trajectories. In order to do so, we employ the differential equations method \cite{DarlingNorris, Wormald}. Here we briefly recall the main idea, closely following Darling and Norris~\cite{DarlingNorris}, although our presentation is for discrete-time rather than continuous-time processes and is in a somewhat simplified setting. Suppose that we are given a stochastic process $P^{n}$ evolving in discrete time with finite state-space $\mathcal{S}^n \subset \mathbb{Z}$. We assume that $P^n$ is either itself Markov or is a co-ordinate of some higher-dimensional Markov process $\mathbf{P}^n$. Let the natural filtration of $\mathbf{P}^n$ be denoted by $(\mathcal{F}_{\mathbf{P}}^{n}(i))_{i \ge 0}$. For each $i\geq{0}$, the process $(M^n_P(i))_{i \ge 0}$ defined by
\begin{equation*}
M^{n}_{P}(i)=P^{n}(i) -P^{n}(0) -\sum_{j=0}^{i-1}\E{P^{n}(j+1)-P^{n}(j)\vert{\mathcal{F}^{n}_{\mathbf{P}}(j)}}
\end{equation*} 
is a martingale. Then, for each fixed $t>0$, we have
\begin{equation*}
\frac{P^{n}(\fl{nt})}{n}=\frac{P^{n}(0)}{n}+\frac{M^{n}_{P}(\fl{nt})}{n}+\sum_{j=0}^{\fl{nt}-1}{\frac{\E{P^{n}(j+1)-P^{n}(j)\vert{{\mathcal{F}^{n}_{\mathbf{P}}(j)}}}}{n}}.
\end{equation*}
Fix $t_0 > 0$ and let $\mathcal{U} \subseteq \R$. We will make four assumptions.
\begin{enumerate}
\item[(1)] For some constant $p(0) \in \mathcal{U}$ we have $\left\vert{\frac{P^{n}(0)}{n}-p(0)}\right\vert \convprob{0}$.
\item[(2)] Suppose that $\nu: [0,t_0] \times \mathcal{U} \to{\mathbb{R}}$ is a bi-Lipschitz function with Lipschitz constant $K$.  Suppose that $p$, the unique solution to the differential equation
\begin{equation*}
\frac{dp(t)}{dt}=\nu(t, p(t))
\end{equation*}
with initial condition $p(0)$, is such that, for some $\epsilon > 0$, $p(t)$ lies at distance greater than $\epsilon$ from $\mathcal{U}^c$ for all $0 \le t \le t_0$.
\item[(3)] Let $T_n = \inf\{t \ge 0: P^n(\fl{nt})/n \notin \mathcal{U}\}$.  We assume that 
\begin{align*}
\sup_{0 \le t \le t_0 \wedge T_n} \left\vert\sum_{j=0}^{\fl{nt}-1}{\frac{\E{P^{n}(j+1)-P^{n}(j)\vert{{\mathcal{F}^{n}_{P}(j)}}}}{n}-\int_{0}^{t}{\nu \left(s, \frac{P^{n}(\fl{ns})}{n}\right)}ds}\right\vert \convprob 0.
\end{align*} 
\item[(4)] We have 
\[
\frac{\E{M^n_P(\fl{n(t_0 \wedge T_n)})^2}}{n^2} \to 0
\]
as $n \to \infty$.
\end{enumerate}

Assumption (1) tells us that the initial condition is well-concentrated on the scale of $n$.  Assumption (3) gives that, conditionally on $P^n$ being in state $\fl{nx}$ at time $\fl{ns}$, the size of the expected increment of $P^{n}$ is close to $\nu(s,x)$. It is then natural to compare $P^{n}/n$ to the solution to the differential equation in (2), as long as $P^n/n$ remains within the set $\mathcal{U}$. 

\begin{prop} \label{prop:fluidlimit}
Under assumptions (1), (2),  (3) and (4), we have
\[
\sup_{0\leq t \leq T}\left|{\frac{P^{n}(\fl{nt})}{n}-p(t)}\right| \convprob 0
\]
as $n \to \infty$.
\end{prop}

\begin{proof}
Since $p$ solves the given differential equation, we may write it in integral form as
\[
p(t) = p(0) + \int_0^t \nu(s, p(s)) ds.
\]
Hence,
\begin{align*}
& \frac{P^n(\fl{nt})}{n} - p(t) \\
& \qquad = \frac{P^{n}(0)}{n} - p(0) +\frac{M^{n}_{P}(\fl{nt})}{n} + \sum_{j=0}^{\fl{nt}-1}{\frac{\E{P^{n}(j+1)-P^{n}(j)\vert{{\mathcal{F}^{n}_{P}(j)}}}}{n}} - \int_0^t \nu(s,p(s)) ds.
\end{align*}
By Doob's $L^2$ inequality, assumption (4) implies that $\sup_{0 \le t_0 \wedge T_n}\frac{|M^{n}_{P}(i)|}{n} \convprob 0$ as $n\to\infty$.  Fix $\delta > 0$ and let
\begin{align*}
\Omega_{n,\delta} & = \left\{ \left\vert{\frac{P^{n}(0)}{n}-p(0)}\right\vert \le \frac{\delta e^{-Kt_0}}{3} \right\} \cap \left\{ \sup_{0\le t \le t_0 \wedge T_n} \frac{|M^{n}_{P}(\fl{nt})|}{n} \le \frac{\delta e^{-Kt_0}}{3} \right\} \\
& \qquad \cap \left\{ \sup_{0 \le t \le t_0 \wedge T_n} \left\vert\sum_{j=0}^{\fl{nt}-1}{\frac{\E{P^{n}(j+1)-P^{n}(j)\vert{{\mathcal{F}^{n}_{P}(j)}}}}{n}-\int_{0}^{t}{\nu \left(s, \frac{P^{n}(\fl{ns})}{n}\right)}ds}\right\vert \le \frac{\delta e^{-Kt_0}}{3} \right\}.
\end{align*}
By the assumptions, we have $\Prob{\Omega_{n,\delta}^c} \to 0$ as $n \to \infty$. On the event $\Omega_{n,\delta}$, we have that for $0 \le t \le t_0 \wedge T_n$,
\begin{align*}
\left| \frac{P^{n}(\fl{nt})}{n}-p(t) \right| & \le \delta e^{-Kt_0} + \int_0^t \left| \nu \left(\frac{P^{n}(\fl{ns})}{n}\right) ds - \nu(p(s))\right| ds \\
& \leq \delta e^{-Kt_0} + K \int_{0}^{t} \left|\frac{P^{n}(\fl{ns})}{n}-p(s)\right| ds ,
\end{align*}
by the Lipschitz property of $\nu$ on $\mathcal{U}$. Hence, by Gronwall's lemma, on the event $\Omega_{n,\delta}$,
\[
\sup_{0 \le t \le t_0 \wedge T_n} \left| \frac{P^{n}(\fl{nt})}{n}-p(t) \right|  \le \delta.
\]
The result follows.
\end{proof}

We begin with a preparatory lemma. 
\begin{lem} \label{Lem Expected increments with nice errors}
For $t<1$ and $0\leq{i}\leq{\lfloor{nt}\rfloor}$ we have
\begin{align*}
\E{U^n(i+1) - U^n(i) | \mathcal{F}^n_i} 
& = -1+E_{U,1}^n(i)\\
\E{(U^n(i+1) - U^n(i))^2 | \mathcal{F}^n_i} 
& =1+E_{U,2}^n(i)\\
\E{V^n(i+1) - V^n(i) | \mathcal{F}^n_i}
& = \frac{-4V^n(i)}{3U^n(i)}+E_{V}^{n}(i)\\
\E{(V^n(i+1) - V^n(i))^2 | \mathcal{F}^n_i}
& = \frac{4V^n(i)}{3U^n(i)}-E_{V}^{n}(i)
\end{align*}
where $\sup_{0\leq{i}\leq{\lfloor{nt}\rfloor}}|E_{U,1}^n(i)|$, $\sup_{0\leq{i}\leq{\lfloor{nt}\rfloor}}|E_{U,2}^n(i)|$, and $\frac{n}{\alpha(n)} \sup_{0\leq{i}\leq{\lfloor{nt}\rfloor}}|E_{V}^{n}(i)|$ all converge to zero in $L^{2}$ 
as $n\to\infty$.
\end{lem}

\begin{proof}
Let us note that 
\[
\E{U^n(i+1) - U^n(i) | \mathcal{F}_i^n} = -1 + \frac{8 V^n(i) + X^n(i) - \I{X^n(i) > 0}}{3 U^n(i) + 4 V^n(i) + X^n(i) - \I{X^n(i) > 0}}.
\]
Having in mind that in each step we remove at most one vertex, for $0\leq{i}\leq{\fl{nt}}$ we have $U^{n}(i)\geq{n-i}$, and so
\[
\sup_{0 \le i \le \fl{nt}} \frac{8 V^n(i) + X^n(i) - \I{X^n(i) > 0}}{3 U^n(i) + 4 V^n(i) + X^n(i) - \I{X^n(i) > 0}} \le \sup_{0 \le i \le \fl{nt}} \frac{8 \alpha(n) + X^n(i)}{3(n-i)},
\]
which converges to 0 in $L^{2}$ as $n \to \infty$ by Lemma~\ref{Lem X is little o of n}. 
Also,
\[
\E{(U^n(i+1) - U^n(i))^2 | \mathcal{F}_i^n} = 1 - \frac{6 V^n(i) + X^n(i) - \I{X^n(i) > 0}}{3 U^n(i) + 4 V^n(i) + X^n(i) - \I{X^n(i) > 0}}
\]
and
\[
\sup_{0 \le i \le \fl{nt}} \frac{6 V^n(i) + X^n(i) - \I{X^n(i) > 0}}{3 U^n(i) + 4 V^n(i) + X^n(i) - \I{X^n(i) > 0}} \le \sup_{0 \le i \le \fl{nt}} \frac{6 \alpha(n) + X^n(i)}{3(n-i)},
\]
which again by Lemma~\ref{Lem X is little o of n} converges to 0 in $L^{2}$ as $n \to \infty$.

Turning now to $V^{n}$, we see that
\begin{align*}
\E{V^n(i+1) - V^n(i) | \mathcal{F}_i^n} & = - \frac{4 V^n(i)}{3U^n(i)} + \frac{4 V^n(i)}{3U^n(i)} \cdot{\frac{4 V^n(i)+X^n(i)-\I{X^n(i)>0}}{3U^n(i)+4V^n(i)+X^n(i)-\I{X^n(i)>0}}},\\
\E{(V^n(i+1) - V^n(i))^2 | \mathcal{F}_i^n} & =  \frac{4 V^n(i)}{3U^n(i)} -\frac{4 V^n(i)}{3U^n(i)} \cdot{\frac{4 V^n(i)+X^n(i)-\I{X^n(i)>0}}{3U^n(i)+4V^n(i)+X^n(i)-\I{X^n(i)>0}}}.
\end{align*}

Letting $E_{V}^{n}(i)=\tfrac{4 V^n(i)}{3U^n(i)} \cdot{\frac{4 V^n(i)+X^n(i)-\I{+X^n(i)>0}}{3U^n(i)+4V^n(i)+X^n(i)-\I{+X^n(i)>0}}}$, we have
\[
\sup_{0 \le i \le \fl{nt}} |E_{V}^n(i)| \le \frac{4\alpha(n)}{3(n-i)} \frac{4 \sup_{0 \le i \le \zeta_n} (X^n(i) + V^n(i))}{3(n-i)}
\]
and so Lemma \ref{Lem X is little o of n} implies that $\sup_{0\leq{i}\leq{\lfloor{nt}\rfloor}}\frac{n}{\alpha(n)}|E_{V}^n(i)| \to 0$ in $L^{2}$ as $n \to \infty$.
\end{proof}

\begin{lem} \label{Lem V3 fluid limit}
Let $t < 1$. Then as $n \to \infty$, we have
\begin{align*}
\sup_{0 \le i \le \fl{nt}} \left|\frac{U^n(i)}{n} -\left(1-\frac{i}{n}\right) \right| & \convprob 0.
\intertext{If, in addition, we assume that $\alpha(n) \to \infty$ as $n \to \infty$, then}
\sup_{0 \le i \le  \fl{nt}} \left| \frac{V^n(i)}{\alpha(n)} - \left(1-\frac{i}{n}\right)^{4/3} \right| & \convprob 0.
\end{align*}
\end{lem}

\begin{proof}
Having in mind the size of expected increments given by Lemma \ref{Lem Expected increments with nice errors}, and $U^{n}(0)=n$ and $V^{n}(0)=\alpha(n)$, the candidate for the fluid limit is the solution to the system of differential equations given by
\begin{equation} \label{eq fluid limit proxy u v}
\begin{aligned}
\frac{du(s)}{ds} &= -1,\\
\frac{dv(s)}{ds} &= -{\frac{4v(s)}{3u(s)}},
\end{aligned}
\end{equation}
for $0\leq{s}\leq{t}$ and the initial condition $u(0) = 1$, $v(0) = 1$. The solution is given by $u(s) = 1-s$, $v(s)=(1-s)^{4/3}$. We observe that $u(s)$ and $v(s)$ remain bounded away from 0 for all $0 \le s \le t$.

Let
\[
T^n_{U} = \inf\{s \ge 0: U^n(\fl{ns}) = 0\}.
\]
Since $U^n(i) \ge n-i$, it is clear that $T^n_U \ge 1$ for all $n$ and so, in particular, $T_n^U > t$.
For $i \ge 0$, let
\[
M_U^n(i) = U^n(i) -U^{n}(0) - \sum_{j=0}^{i-1} \E{U^n(j+1)-U^n(j) | \mathcal{F}_j^n},
\]
so that $(M_U^n(i))_{i \ge 0}$ is a martingale.  By Lemma~\ref{Lem Expected increments with nice errors}, we have that for $0 \le s \le t$,
\[
\left| \sum_{j=0}^{\fl{ns}-1} \frac{\E{U^n(j+1)-U^n(j) | \mathcal{F}_j^n}}{n} + \int_0^s du \right| 
\le  \sup_{0\leq{s}\leq{t}}\left| \frac{\fl{ns}}{n}- s \right| + t \sup_{0 \le i \le \fl{nt}} |E^n_{U,1}(i)|  \convprob 0,
\]
as $n \to \infty$. Moreover,
\[
\E{M_U^n(i)^2} = \E{\sum_{j=0}^{i-1} \var{U^n(j+1) - U^n(j) | \mathcal{F}_j^n}} \le i \left(1 + \E{\sup_{0 \le j \le i} |E_{U,2}^n(j)|}\right)
\]
and so by Lemma \ref{Lem Expected increments with nice errors} we have
\begin{equation} \label{eqn:MU bound}
\frac{\E{M_U^n(\fl{nt)})^2}}{n^2} \le \frac{\fl{nt}(1 + \E{\sup_{0 \le i \le \fl{nt}} |E^n_{U,2}(i)|} )}{n^2} \to 0.
\end{equation}
Hence, by Proposition~\ref{prop:fluidlimit},
\[
\sup_{0 \le s \le t} \left|\frac{U^n(\fl{ns})}{n} - u(s)\right| \convprob 0.
\]
We now turn to $V^{n}$. For $i\geq{0}$, let
\[
M_V^n(i) = V^n(i) -V^{n}(0) - \sum_{j=0}^{i-1} \E{V^n(j+1)-V^n(j) | \mathcal{F}_j^n},
\]
be the standard martingale. Let us deal first with its second moment.  We have
\[
\E{M_V^n(i)^2} = \E{\sum_{j=0}^{i-1} \var{V^n(j+1) - V^n(j) | \mathcal{F}_j^n}} \le i \cdot {\sup_{0\leq{j}\leq{i}}{\left({\left|{\frac{4V^{n}(j)}{3U^{n}(j)}}\right|+\left|{E^{n}_{V}(i)}\right|}\right)}}.
\]
Since $\frac{4V^{n}(j)}{3U^{n}(j)}\leq{\frac{4\alpha(n)}{3(n-j)}}$, Lemma \ref{Lem Expected increments with nice errors} gives us
\begin{equation} \label{eqn:MV bound}
 \frac{\E{M_V^n(\fl{nt})^2}}{\alpha(n)^2} \le \frac{\fl{nt}\left(\frac{4\alpha(n)}{3(n-\fl{nt})} + \E{\sup_{0 \le i \le \fl{nt}} |E^n_{V}(i)|} \right)}{\alpha(n)^2} \to 0.
\end{equation}
Turning now to the drift, we have
\begin{align*}
& \sup_{0\leq{s}\leq{t}}\left| \sum_{j=0}^{\fl{ns} - 1} \frac{\E{V^{n}(j+1)-V^{n}(j) |  \mathcal{F}^{n}_j}}{\alpha(n)}+ \int_0^s \frac{4v(r)}{3u(r)} dr \right| \\
& \qquad \le \sup_{0\leq{s}\leq{t}}\left|{\sum_{j=0}^{\fl{ns}-1}{\frac{4V^{n}(i)/\alpha(n)}{3U^{n}(i)}}-\int_{0}^{s}{\frac{4v(r)}{3u(r)}}dr}\right| + \frac{nt}{\alpha(n)} \sup_{0 \le i \le \fl{nt}} |E_V^n(i)| \\
& \qquad \le \sup_{0\leq{s}\leq{t}}\left| \int_{0}^{\fl{ns}/n} \left( \frac{4V^{n}(\fl{nr})/\alpha(n)}{3U^{n}(\fl{nr})/n} - \frac{4v(r)}{3u(r)} \right) dr \right| + \sup_{0 \le s \le t} \left| \int_{\fl{ns}/n}^s{\frac{4v(r)}{3u(r)}}dr \right| \\
& \qquad \qquad + \frac{nt}{\alpha(n)} \sup_{0 \le i \le \fl{nt}} |E_V^n(i)|.
\end{align*}
The penultimate term on the right-hand side clearly tends to 0, and the last term converges to 0 in probability by Lemma~\ref{Lem Expected increments with nice errors}.  We have
\begin{align*}
& \sup_{0\leq{s}\leq{t}}\left| \int_{0}^{\fl{ns}/n} \left( \frac{4V^{n}(\fl{nr})/\alpha(n)}{3U^{n}(\fl{nr})/n} - \frac{4v(r)}{3u(r)} \right) dr \right| \\
& \le \int_{0}^{\fl{nt}/n} \frac{4V^{n}(\fl{ns})/\alpha(n)}{3u(s)} \cdot{\left|{\frac{u(s)}{U^{n}(\fl{ns})/n}}-1\right|}ds
+ \int_{0}^{\fl{nt}/n}\frac{4}{3u(s)} \left|{\frac{V^{n}(\fl{ns})}{\alpha(n)}-v(s)}\right|ds \\
& \le \frac{4}{3(1-t)} \sup_{0 \le s \le t} \left| \frac{u(s)}{U^n(\fl{ns})/n} - 1 \right| + \frac{4}{3(1-t)} \int_0^t \left| {\frac{V^{n}(\fl{ns})}{\alpha(n)}-v(s)}\right|ds.
\end{align*}
The first term on the right-hand side converges to 0 in probability.  So following the argument in the proof of Proposition~\ref{prop:fluidlimit}, we easily obtain
\begin{equation*}
\sup_{0\leq{s}\leq{t}}\left|{\frac{V^{n}(\fl{ns})}{\alpha(n)}-v(s)}\right| \convprob 0. \qedhere
\end{equation*}
\end{proof}

%
%
\section{Fluid limit for $\alpha(n) > > \sqrt{n}$} \label{Sect Fluid limit X N large omega}

In the case where $\alpha(n) > > \sqrt{n}$, we will also prove fluid limits for $X^n$ and $N^n$.  

\begin{thm} \label{thm fluid limit for X for omega>>sqrt n}
Fix $t \in (0,1)$. For $\alpha(n)\gg\sqrt{n}$, as $n\to\infty$ we have
\begin{equation*}
\begin{aligned}
\sup_{0\le s \le t} \left| \frac{X^{n}(\fl{ns})}{\alpha(n)}-\left((1-s)^{2/3}-(1-s)^{4/3}\right) \right| &\convprob 0,\\
\sup_{0\le s \le t} \left|\frac{N^{n}(\fl{ns})}{\alpha(n)}-\left(\frac{1}{4}-\frac{1}{2}(1-s)^{2/3}+\frac{1}{4}(1-s)^{4/3}\right)\right|& \convprob 0, \\
\frac{L^{n}(\fl{nt})}{\alpha(n)}  & \convprob 0.
\end{aligned}
\end{equation*}
\end{thm}

To this end, we again study the expected increments.

\begin{lem} \label{Lem expected increments of X} \label{Lem Exp increments N}
For $t<1$ and $0\leq{i}\leq{\lfloor{nt}\rfloor}$ we have
\begin{align*}
\E{X^n(i+1) - X^n(i) | \mathcal{F}^n_i}
& = 2\I{X^n(i)=0}+\frac{2\alpha(n)(1-i/n)^{4/3}-2X^n(i)}{3(n-i)}+E_{X,1}^{n}(i),\\
\E{(X^n(i+1) - X^n(i))^2 | \mathcal{F}^n_i}
& = 4\I{X^n(i)=0}+1+E_{X,2}^n(i),\\
\E{N^n(i+1) - N^n(i) | \mathcal{F}^n_i}&=\E{(N^n(i+1) - N^n(i))^2 | \mathcal{F}^n_i}=
 \frac{X^n(i)}{3(n-i)}+E_N^{n}(i),
\end{align*}
where $\frac{n}{\alpha(n) \vee \sqrt{n}} \sup_{0\leq{i}\leq{\lfloor{nt}\rfloor}}|E_{X,1}^{n}(i)|\convprob{0}$, and $\sup_{0\leq{i}\leq{\lfloor{nt}\rfloor}}|E_{X,2}^{n}(i)| \to {0}$ and $\frac{n}{\alpha(n) \vee \sqrt{n}} \sup_{0\leq{i}\leq{\lfloor{nt}\rfloor}} |E_N^{n}(i)| \to 0$ in $L^2$ as $n\to\infty$.
\end{lem}

\begin{proof}
By (\ref{eq exp changes X}),  we have
\begin{align*}
& \E{X^n(i+1) - X^n(i) | \mathcal{F}_i^n} \\
& = 2\I{X^n(i)=0}+\frac{2\alpha(n)(1-i/n)^{4/3}-2X^n(i)}{3(n-i)}+\frac{2V^n(i)-2\alpha(n)(1-i/n)^{4/3}+ 2\I{X^n(i)>0}}{3(n-i)}\\
& \qquad +\frac{2V^n(i)-2X^n(i)+2\I{X^n(i)>0}}{3(n-i)}\cdot{\left(\left({1+\tfrac{3U^n(i)-3(n-i)+4V^n(i)+X^n(i)-\I{X^n(i)>0}}{3(n-i)}}\right)^{-1}-1\right)}\\
&=: 2\I{X^n(i)=0}+\frac{2\alpha(n)(1-i/n)^{4/3}-2X^n(i)}{3(n-i)}+E_{X,1}^{n}(i).
\end{align*}
For $0\leq{i}\leq{\lfloor{nt}\rfloor}$, we have
\begin{equation} \label{eqn:probto0}
\begin{aligned}
& \frac{n}{\alpha(n)} \sup_{0 \le i \le \fl{nt}} \left| \frac{2V^n(i)-2\alpha(n)(1-i/n)^{4/3}+ 2\I{X^n(i)>0}}{3(n-i)} \right| \\
&\qquad \le \frac{2}{3(1-t)} \sup_{0 \le i \le \fl{nt}} \left| \frac{V^n(i)}{\alpha(n)} - \left(1 - \frac{i}{n}\right)^{4/3} \right| + \frac{2}{3(1-t)\alpha(n)} \convprob 0,
\end{aligned}
\end{equation}
by the fluid limit for $V^n$ in Lemma~\ref{Lem V3 fluid limit}. Since $(1+x)^{-1} \ge 1-x$ for $x > 0$, and $U^n(i) \ge n-i$, we have
\begin{equation} \label{eq the most useful estimate for errors}
\begin{aligned}
& \sup_{0 \le i \le \fl{nt}} \left|{\left({1+\tfrac{3U^n(i)-3(n-i)+4V^n(i)+X^n(i)-\I{X^n(i)>0}}{3(n-i)}}\right)^{-1}-1}\right|\\
&\qquad \leq{ \sup_{0 \le i \le \fl{nt}} \left|{\frac{3U^n(i)-3(n-i)+4V^n(i)+X^n(i)-\I{X^n(i)>0}}{3(n-i)}}\right|}\\
&\qquad \leq{\frac{1}{3(1-t)}} \left(3\sup_{0\le i \le \fl{nt}} \left|{\frac{U^n(i)}{n}- \left( 1-\frac{i}{n} \right) }\right|+\frac{4\alpha(n)}{n}+{{\frac{1+ \sup_{0\le i \le t}{X^n({i})}}{n}}} \right) \convprob 0
\end{aligned}
\end{equation}
by Lemma~\ref{Lem X is little o of n} (which says that $X^n$ is of smaller order than $n$), and the fluid limit for $U^n$ from Lemma \ref{Lem V3 fluid limit}.

By Lemma~\ref{Lem X is little o of n}, we have that
\[
\frac{1}{\alpha(n) \vee \sqrt{n}} \sup_{0 \le i \le \fl{nt}} (X^n(i) + V^n(i)) 
\]
is tight in $n$.  It follows that
\[
\frac{n}{\alpha(n) \vee \sqrt{n}} \sup_{0 \le i \le \fl{nt}} \frac{2V^n(i) - 2X^n(i) + 2 \I{X^n(i) > 0}}{3(n-i)}
\]
is also tight, and so together with (\ref{eqn:probto0}) and (\ref{eq the most useful estimate for errors}) we get $\frac{n}{\alpha(n) \vee \sqrt{n}} \sup_{0 \le i \le \fl{nt}} |E_{X,1}^n(i)| \convprob 0$.

For the second moment, rewriting the expression (\ref{eqn:squareincrX}) we have
\begin{align*}
& \E{(X^n(i+1) - X^n(i))^2 | \mathcal{F}^n_i} \\
& \qquad = 4\I{X^n(i)=0}+1+\frac{6V^n(i)+8V^n(i)\I{X^n(i)=0}+3X^n(i)-3\I{X^n(i)>0}}{3U^n(i)+4V^n(i)+X^n(i)-\I{X^n(i)>0}} \\
& \qquad =: 4\I{X^n(i)=0}+1 + E_{X,2}^n(i).
\end{align*}
But then
\[
\sup_{0 \le i \le \fl{nt}} |E_{X,2}^n(i)| \leq{\frac{14\alpha(n)+3\sup_{0\leq{i}\leq{\lfloor{nt}\rfloor}}X^n(i)}{3n(1-t)}}
\]
which converges to 0 in probability and in $L^2$, by Lemma~\ref{Lem X is little o of n}.
Turning now to the increments of $N^{n}$, we have
\begin{align*}
& \E{N^n(i+1) - N^n(i) | \mathcal{F}^n_i} \\
& \qquad =\E{(N^n(i+1) - N^n(i))^2 | \mathcal{F}^n_i}\\
& \qquad =  \frac{X^n(i)}{3U^{n}(i)+4V^{n}(i)+X^n(i)-\I{X^n(i)>0}}\\
 & \qquad \qquad +\frac{3U^n(i)-2\I{X^{n}(i)>0}(3U^{n}(i)+4V^{n}(i)+X^{n}(i)-3)}{2(3U^{n}(i)+4V^{n}(i)+X^n(i)-\I{X^n(i)>0})(3U^{n}(i)+4V^{n}(i)+X^n(i)-\I{X^n(i)>0}-2)}\\
 & \qquad = \frac{X^n(i)}{3(n-i)}+\frac{X^n(i)}{3(n-i)}\left({{\left({1+\tfrac{3U^n(i)-3n(1-i/n)+4V^n(i)+X^n(i)-\I{X^n(i)>0}}{3(n-i)}}\right)^{-1}}-1}\right)\\
 & \qquad \qquad +\frac{3U^{n}(i)-2\I{X^{n}(i)>0}(3U^{n}(i)+4V^{n}(i)+X^{n}(i)-3)}{2(3U^{n}(i)+4V^{n}(i)+X^{n}(i)-\I{X^{n}(i)>0})(3U^{n}(i)+4V^{n}(i)+X^{n}(i)-\I{X^{n}(i)>0}-2)}\\
 & \qquad =: \frac{X^n(i)}{3(n-i)} + E_N^{n}(i).
\end{align*}
For $0\leq{i}\leq{\lfloor{nt}\rfloor}$ we have
\begin{equation*}
\begin{aligned}
&\left|{\frac{3U^{n}(i)-2\I{X^{n}(i)>0}(3U^{n}(i)+4V^{n}(i)+X^{n}(i)-3)}{2(3U^{n}(i)+4V^{n}(i)+X^{n}(i)-\I{X^{n}(i)>0})(3U^{n}(i)+4V^{n}(i)+X^{n}(i)-\I{X^{n}(i)>0}-2)}}\right|\\
&\leq{\frac{3}{{2(3U^{n}(i)+4V^{n}(i)+X^{n}(i)-\I{X^{n}(i)>0})}}}\\
&\leq{\frac{3}{2n(1-t)}}.
\end{aligned}
\end{equation*}
Hence,
\[
\begin{aligned}
& \frac{n}{\alpha(n) \vee \sqrt{n}} \sup_{0 \le i \le \fl{nt}} |E_N^n(i)|  \\
& \qquad \le \frac{3}{2(1-t) (\alpha(n)\vee \sqrt{n})} \\
& \qquad \qquad + \frac{\sup_{0 \le i \le \fl{nt}}X^n(i)}{3 (1-t) (\alpha(n) \vee \sqrt{n})} \left({{\left({1+\tfrac{3U^n(i)-3n(1-i/n)+4V^n(i)+X^n(i)-\I{X^n(i)>0}}{3(n-i)}}\right)^{-1}}-1}\right).
\end{aligned}
\]
Using (\ref{eq the most useful estimate for errors}) and the fact that $\frac{1}{\alpha(n) \vee \sqrt{n}} \sup_{0 \le i \le \fl{nt}} X^n(i)$ is bounded in $L^2$, this gives that 
\[
\frac{n}{\alpha(n) \vee \sqrt{n}} \sup_{0 \le i \le \fl{nt}} |E_N^n(i)| \to 0,
\]
in $L^2$, as desired.
\end{proof}

Before we can proceed to the proof of Theorem \ref{thm fluid limit for X for omega>>sqrt n}, we need some technical lemmas to deal with the fact that $X^n$ reflects off 0.  Let 
\[
\gamma^n = \inf \{i \ge 0: X^n(i) \ge V^n(i)\}.
\]

\begin{lem} \label{lem:slemma}
Fix $t \in(1/2, 1)$.  Then there exists $t_0 \in (0,t)$ sufficiently small that
\[
\Prob{\gamma^n < t_0 n} \to 0
\]
as $n \to \infty$.
\end{lem}

\begin{proof}
Let $\tilde{\gamma}^n = \inf\{i \ge 0: X^n(i) \ge 3 \alpha(n)/8\}$.  Observe that $V^n$ is decreasing and that if $s_0 = 1 - \frac{1}{2^{3/4}} \approx 0.405$ then $(1-s_0)^{4/3} = 1/2 > 3/8$.  Then by Lemma~\ref{Lem V3 fluid limit},
\[
\Prob{V^n(\fl{ns_0}) > 3\alpha(n)/8} \to 1
\]
as $n \to \infty$. So if we can show that there exists $s_1 \in (0,t)$ such that
\begin{equation} \label{eqn:simpler}
\Prob{\tilde{\gamma}^n < ns_1} \to 0
\end{equation}
as $n \to \infty$ then, setting $t_0 = s_0 \wedge s_1$, we obtain
\[
\Prob{\gamma^n < nt_0} \to 0.
\]
So it remains to prove (\ref{eqn:simpler}) for some $s_1 \in (0,t)$.

The time-inhomogeneity of the process $X^n$ makes explicit calculations awkward.  We instead couple $X^n$ with a simpler process $A^n$. If $A^n(i) = 0$ then $A^n(i+1) = 1$.  If $A^n(i) > 0$ then
\[
A^n(i+1) - A^n(i) = \begin{cases}
-1 & \text{ with probability } \frac{1}{2} \\
+1 & \text{ with probability } \frac{1}{2} - \frac{4 \alpha(n)}{3n(1-t)} \\
+2 & \text{ with probability } \frac{4\alpha(n)}{3n(1-t)}.
\end{cases}
\]
(We implicitly take $n$ sufficiently large that $\frac{4 \alpha(n)}{3n(1-t)} < 1/2$.)  We set $A^n(0) = 0$.
It is straightforward to see that we may couple $X^n$ and $A^n$ in such a way that $X^n(i) \le 3 + A^n(i)$ for all $0 \le i \le \fl{nt}$.  Let $\lambda_n = \inf\{i \ge 0: A^n(i) \ge 3\alpha(n)/8 - 3\}$.  Then we certainly have $\lambda_n \le \tilde{\gamma}_n$.  So we will find $s_1 \in (0,t)$ such that
\[
\Prob{\lambda_n < n s_1} \to 0.
\]
For simplicity, write $b = \fl{3\alpha(n)/8} - 3$ for the upper barrier and $d = \frac{4\alpha(n)}{3n(1-t)}$ for the drift.  The quantity 
\[
h_1 = \Prob{\text{$A^n$ hits $0$ before $\{b,b+1\}$} | A^n(0) = 1}
\]
is calculated via a standard calculation in Lemma~\ref{lem:hittingprob} in the Appendix.
Substituting the given values of $b$ and $d$ in the (lengthy) expression there, we obtain that
\[
1-h_1 \sim \frac{8 \alpha(n)}{3n(1-t)},
\]
as $n \to \infty$. 
By the strong Markov property, the random walk $A^n$ visits 0 a Geometric($1-h_1$) number of times before going above $3\alpha(n)/8 - 3$, and so
\[
\frac{1}{\alpha(n)} \E{\sum_{i=0}^{\lambda_n-1}\I{A^n(i) = 0}} \sim \frac{3n(1-t)}{8 \alpha(n)^2} \to 0,
\]
since $\alpha(n) > > \sqrt{n}$.

Now let
\[
M^n_{A}(i) := A^n(i) -\frac{4 \alpha(n) i}{3n(1-t)} - \sum_{j=0}^{i-1} \left(1 - \frac{4 \alpha(n)}{3n(1-t)} \right) \I{A^n(j) = 0}.
\]
Then $(M^n_{A}(i))_{i \ge 0}$ is a martingale.  Let $\mathcal{F}^{n}_A(j)$ be the natural filtration of $A^{n}$. We have
\[
\var{A^n(j+1)-A^n(j)|\mathcal{F}_A^n(j)} \le \left(1 + \frac{4\alpha(n)}{n(1-t)}\right) \I{A^n(j) > 0} \le 1 + \frac{4\alpha(n)}{n(1-t)}.
\]
and so
\[
\E{(M^n_{A}(i))^2} \le i \left(1 + \frac{4\alpha(n)}{n(1-t)}\right) \le 2i
\]
for $n$ sufficiently large. For any $s > 0$, then,
\[
\sup_{0 \le i \le \lambda_n \wedge \fl{ns}} \left| \frac{A^n(i)}{\alpha(n)} - \frac{4 i}{3n(1-t)} \right| \le \frac{\sup_{0 \le i \le \fl{ns}} |M^n_{A}(i)|}{\alpha(n)} + \frac{1}{\alpha(n)} \sum_{j=0}^{\lambda_n-1} \I{A^n(i) = 0}.
\]
For any $\delta > 0$, it then follows by using Doob's $L^2$ inequality that
\begin{align*}
\Prob{\sup_{0 \le i \le \lambda_n \wedge \fl{ns}} \left| \frac{A^n(i)}{\alpha(n)} - \frac{4 i}{3n(1-t)} \right| > \delta} & \le 4 \frac{\E{M_{A}^n(\fl{ns})^2}}{\delta^2 \alpha(n)^2} + \frac{1}{\delta \alpha(n)} \E{\sum_{j=0}^{\lambda_n-1} \I{A^n(i) = 0}} \\
& \le \frac{8 \fl{ns}}{\delta^2 \alpha(n)^2} + \frac{1}{\delta \alpha(n)} \E{\sum_{j=0}^{\lambda_n-1} \I{A^n(i) = 0}} \to 0,
\end{align*}
as $n \to \infty$. 

Finally, if we take $s_1 = 3(1-t)/16$ then
\[
\frac{4 \fl{ns_1}}{3n(1-t)} \le \frac{1}{4} < \frac{3}{8}.
\]
Then for $\delta \in (0, 1/8)$,
\[
\Prob{\lambda_n < ns_1} \le \Prob{\sup_{0 \le i \le \lambda_n \wedge \fl{ns_1}} \left| \frac{A^n(i)}{\alpha(n)} - \frac{4 i}{3n(1-t)} \right| > \delta} \to 0
\]
as $n \to \infty$.  The result follows.
\end{proof}

\begin{prop} \label{prop:localtime}
For the $t_0 \in (0,t)$ from Lemma~\ref{lem:slemma}, 
\[
\frac{L^n(\fl{nt_0})}{\alpha(n)} \convprob 0,
\]
as $n\to\infty$.
\end{prop}

\begin{proof}

We again make use of a coupling, this time to produce a lower bound: we define a process $D^n$ such that $D^n(i) \le X^n(i)$ for  $i \le \gamma^n -1$.  Conditionally on $(U^n(i), V^n(i), X^n(i), D^n(i)) = (u,v,x,d)$ with $D^n(i) = d > 0$, we let
\[
D^n(i+1) - D^n(i) = \begin{cases}
+2 & \text{ with probability } \frac{4(x-1)\I{x > 0}}{2(3u+4v+(x-1)\I{x>0})} \\
+1 & \text{ with probability } \frac{3u + 4v - 4(x-1)\I{x > 0}}{2(3u+4v+(x-1)\I{x>0})} \\
-1 & \text{ with probability } \frac{3u + 4v}{2(3u+4v+(x-1)\I{x>0})} \\
-2 & \text{ with probability } \frac{2(x-1)\I{x > 0}}{2(3u+4v+(x-1)\I{x>0})}.
\end{cases}
\]
Conditionally on $(U^n(i), V^n(i), X^n(i), D^n(i)) = (u,v,x,d)$ with $D^n(i) = 0$, let $D^n(i+1) = 1$.  Since we assume $i \le \gamma^n-1$, it is straightforward to see that we may produce a coupling such that $D^n(i) \le X^n(i)$.  Let
\[
\tilde{L}^n(i) = 2 \sum_{j=0}^{i-1} \I{D^n(j) = 0},
\]
and observe that we have $\tilde{L}^n(\gamma^n \wedge \fl{nt_0}) \ge L^n(\gamma^n \wedge \fl{nt_0})$.  

But then for any $\delta > 0$,
\[
\Prob{\frac{L^n(\fl{nt_0})}{\alpha(n)} > \delta} \le \Prob{\frac{\tilde{L}^n(\gamma^n \wedge \fl{nt_0})}{\alpha(n)} > \delta, \gamma^n \ge \fl{nt_0}} + \Prob{\gamma^n < \fl{nt_0}}.
\]

By Lemma~\ref{lem:slemma}, we have that $\Prob{\gamma^n < \fl{nt_0}} \to 0$ and so it suffices to prove that the first term tends to 0.

We have
\[
\E{D^n(i+1) - D^n(i) | \mathcal{F}^n_D(i)} = \I{D^n(i) = 0},
\]
where $(\mathcal{F}_D^n(i))_{i \ge 0}$ denotes the natural filtration of $(U^n, V^n, X^n, D^n)$. Also, 
\[
\var{D^n(i+1) - D^n(i) | \mathcal{F}_D^n(i)} = 1 + \frac{9(X^n(i) - \I{X^n(i) > 0})}{3U^n(i) + 4V^n(i) + X^n(i) - \I{X^n(i) > 0}},
\]
where
\[
\frac{9(X^n(i) - \I{X^n(i) > 0})}{3U^n(i) + 4V^n(i) + X^n(i) - \I{X^n(i) > 0}} \le \frac{3 \alpha(n)}{n(1-s)},
\]
for $i \le \gamma^n \wedge \fl{ns}$. 

Hence,
\[
\left(D^n(i) - \sum_{j=0}^{i-1}\I{D^n(j)=0} \right)_{i \ge 0}
\]
is a mean 0 martingale, with bounded steps and
\[
\sum_{j=0}^{i-1} \E{\left(D^n(j+1) - D^n(j) - \I{D^n(j) = 0}\right)^2 \Big| \mathcal{F}^n_D(j)} = i + E^{n}_{D}(i)
\]
for $i \le \gamma^n \wedge \fl{nt_0}$, where $\sup_{i \le \gamma^n \wedge \fl{nt_0}}\frac{{E^{n}_{D}(i)}}{\alpha(n)}$ is bounded with high probability.  By the martingale functional central limit theorem (Theorem 1.4, Section 7.1 of \cite{EthierKurtz}), we then have that, on the event $\{\gamma^n \ge \fl{nt_0}\}$,
\[
\frac{1}{\sqrt{n}} \left(D^n(\fl{nu}) - \sum_{j=0}^{\fl{nu}-1}\I{D^n(j)=0} \right)_{0 \le u \le t_0} \convdist (W_u)_{0 \le u \le t_0},
\]
where $W$ is a standard Brownian motion.  But we also have that
\[
\frac{1}{\sqrt{n}} \sum_{j=0}^{\fl{nt_0} - 1} \I{D^n(j) = 0} = - \inf_{0 \le i \le \fl{nt_0}}\frac{1}{\sqrt{n}} \left(D^n(i) - \sum_{j=0}^{i-1}\I{D^n(j)=0} \right),
\]
and so by the continuous mapping theorem we deduce that on $\{\gamma^n \ge \fl{nt_0}\}$,
\[
\frac{\tilde{L}^n(\fl{nt_0})}{\sqrt{n}} = \frac{2}{\sqrt{n}} \sum_{j=0}^{\fl{nt_0} -1} \I{D^n(j) = 0} \convdist - 2 \inf_{0 \le u \le t_0} W_u.
\]
The result then follows since $\alpha(n) > > \sqrt{n}$.
\end{proof}

We now turn to the proof of the main theorem in this section.

\begin{proof}[Proof of Theorem~\ref{thm fluid limit for X for omega>>sqrt n}]

Lemma \ref{Lem expected increments of X} gives us that the expected increments of $X^n$ and $N^n$ are 
\begin{align*}
\E{X^n(i+1) - X^n(i) | \mathcal{F}^n_i}
& =2 \I{X^n(i) = 0} + \frac{\alpha(n)}{n} \left(\frac{2(1-i/n)^{1/3}}{3}-\frac{2X^n(i)/\alpha(n)}{3(1-i/n)}\right)+E_{X,1}^{n}(i),\\
\E{N^n(i+1) - N^n(i) | \mathcal{F}^n_i}&=  \frac{\alpha(n)}{n} \frac{X^n(i)/\alpha(n)}{3(1-i/n)}+E_N^{n}(i).
\end{align*}
This suggests that, as long as $X^n$ does not hit 0 too often, the candidate for the fluid limit should be the solution to
\begin{equation}\label{eq fluid limit proxy XN}
\frac{dx(s)}{ds} =\frac{2}{3}(1-s)^{1/3}-\frac{2x(s)}{3(1-s)}, \quad \frac{dm(s)}{ds} =\frac{x(s)}{3(1-s)},
\end{equation}
with initial conditions $x(0)=0$, $m(0)=0$. The unique solution to this system of differential equations is given by $x(s)=(1-s)^{2/3}-(1-s)^{4/3}$ and $m(s)=\frac{1}{4}-\frac{1}{2}(1-s)^{2/3}+\frac{1}{4}(1-s)^{4/3}$ for $s\in[0,1]$. We will use Proposition~\ref{prop:localtime} to help control the local time at the start, and then prove the fluid limit result in the standard manner.

Let
\[
T^n = \inf\{s \ge t_0: X^n(\fl{ns}) = 0\}.
\]
Then for $t \ge t_0$ we have
\[
L^n(\fl{n(t \wedge T^n)}) = L^n(\fl{nt_0}),
\]
since we may only accumulate local time at 0. For $i \ge 0$, let
\[
M_X^n(i) = X^n(i) - \sum_{j=0}^{i-1} \E{X^n(j+1)-X^n(j) | \mathcal{F}_j^n},
\]
so that $(M_X^n(i))_{i \ge 0}$ is a martingale.  We have
\[
\E{M_X^n(i)^2} = \E{\sum_{j=0}^{i-1} \var{X^n(j+1) - X^n(j) | \mathcal{F}_j^n}} \le i \left(1 + \E{\sup_{0 \le j \le i} |E_{X,2}^n(j)|}\right),
\]
and so 
\[
\frac{\E{M_X^n(\fl{nt})^2}}{\alpha(n)^2} \le \frac{\fl{nt}\left(1 + \E{\sup_{0 \le i \le \fl{nt}} |E^n_{X,2}(i)|} \right)}{\alpha(n)^2} \to 0,
\]
by Lemma~\ref{Lem expected increments of X} and the fact that $\alpha(n) > > \sqrt{n}$.  It follows as usual that
\begin{equation} \label{eqn:Mbound}
\frac{\sup_{0 \le s \le t \wedge T_n} |M_X^n(\fl{nt})|}{\alpha(n)} \convprob 0.
\end{equation}
Then 
\begin{align*}
& \sup_{0 \le s \le t \wedge T^n}\left| \sum_{j=0}^{\fl{ns} -1} \frac{\E{X^n(j+1) - X^n(j) | \mathcal{F}_j^n}}{\alpha(n)} - \int_0^s \left(\frac{2}{3} (1-u)^{1/3} - \frac{2x(s)}{3(1-s)} \right) du \right| \\
& \le \sup_{0 \le s \le t } \left| \frac{1}{n} \sum_{j=0}^{\fl{ns}-1} \frac{2}{3} (1-j/n)^{1/3} - \int_0^s \frac{2}{3} (1-u)^{1/3} du \right|  \\
& \qquad +  \sup_{0 \le s \le t \wedge T^n}\left| \frac{1}{n} \sum_{j=0}^{\fl{ns} -1} \frac{2}{3(1-j/n)} \frac{X^n(j)}{\alpha(n)} - \int_0^s \frac{2x(u)}{3(1-u)}du \right| + \frac{\fl{nt}}{\alpha(n)} \sup_{0 \le i \le \fl{n(t\wedge T^n)}} |E^n_{X,1}(i)|.
\end{align*}
The first term is clearly the difference between a Riemann approximation to an integral and that integral, and tends to 0.  The third term converges in probability to 0 by Lemma~\ref{Lem expected increments of X}.  The middle term is bounded above by
\begin{align*}
& \frac{2}{3(1-t)} \int_0^{\fl{n(t \wedge T^n)}/n} \left| \frac{X^n(\fl{nu})}{\alpha(n)} - x(u) \right| du + \int_0^t \left|\frac{2}{3(1-u)} -  \frac{2}{3(1-\fl{nu}/u)} \right| x(u) du \\
& \qquad + \sup_{0 \le s \le t} \int_{\fl{ns}/n}^s \frac{2x(u)}{3(1-u)} du,
\end{align*}
where the second and third terms clearly tend to 0 as $n \to \infty$.  The rest of the argument then goes through as in the proof of Proposition~\ref{prop:fluidlimit} to show that
\[
\sup_{0 \le s \le t \wedge T^n} \left|\frac{X^n(\fl{ns})}{\alpha(n)} - x(s)\right| \convprob 0.
\]
Now observe that for fixed $t > t_0$, $\inf_{t_0 \le s \le t} x(s) > 0$.  So for sufficiently small $\delta > 0$,
\[
\sup_{0 \le s \le t \wedge T^n} \left|\frac{X^n(\fl{ns})}{\alpha(n)} - x(s)\right| \le \delta
\]
implies that $T^n > t$ and $L^n(\fl{nt}) = L^n(\fl{nt_0})$.  Hence,
\[
\sup_{0 \le s \le t} \left|\frac{X^n(\fl{ns})}{\alpha(n)} - x(s)\right| \convprob 0 \quad \text{ and } \quad \frac{L^n(\fl{nt})}{\alpha(n)} \convprob 0.
\]

Finally, to get the result for $N^n$, note that again
\[
M^n_N(i) = N^n(i) - \sum_{j=0}^{i-1} \E{N^n(j+1)-N^n(j)|\mathcal{F}_j^n},
\]
defines a martingale, with
\[
\E{M^n_N(i)^2} \le \E{\sum_{j=0}^{i-1} \E{N^n(j+1)-N^n(j) | \mathcal{F}_j^n}}
\]
and so
\begin{align*}
\frac{\E{M^n_N(\fl{nt})^2}}{\alpha(n)^2} & \le \frac{1}{n\alpha(n)} \E{ \sum_{i=0}^{\fl{nt}-1} \frac{X^n(i)/\alpha(n)}{3(1-i/n)}} + \frac{nt}{\alpha(n)^2}\E{\sup_{0 \le i \le \fl{nt}} E_N^n(i)} \\
& \le \frac{t}{3(1-t)\alpha(n)^2}\E{\sup_{0 \le i \le \fl{nt}}X^n(i)} + \frac{nt}{\alpha(n)^2}\E{\sup_{0 \le i \le \fl{nt}} E_N^n(i)} \to 0,
\end{align*}
by Lemmas~\ref{Lem X is little o of n} and \ref{Lem expected increments of X}. Then
\begin{align*}
\sup_{0 \le s \le t} \left|\frac{N^n(\fl{ns})}{\alpha(n)} - m(s) \right| & \le \frac{\sup_{0 \le s \le t} |M_N^n(\fl{ns})|}{\alpha(n)}  + \frac{\sup_{0 \le i \le \fl{nt}} |E_N^n(i)|}{\alpha(n)} \\
& \qquad + \sup_{0 \le s \le t} \left| \frac{1}{n} \sum_{j=0}^{\fl{ns}-1} \frac{1}{3(1-j/n)} \frac{X^n(j)}{\alpha(n)} - \int_0^s \frac{x(s)}{3(1-s)} ds \right|,
\end{align*}
and the usual argument allows us to complete the proof.
\end{proof}

%
%
\section {Diffusion limit for $\alpha(n)=O({\sqrt{n}})$} \label{Sect Diffusion limit small omega}
\subsection{The limiting process}

\begin{prop} \label{Prop sistem 1}
Let $a \ge 0$ and let $(B_t, 0\leq{t} \le 1)$ be a Brownian motion. There is a unique solution $(X^a, L^a)$ to the SDE with reflection at $0$ given by
\begin{equation} \label{eq sys sde omega}
dX_t^a=-\frac{2}{3}{\frac{X_t^a}{(1-t)}}dt+\frac{2a}{3}(1 - t)^{1/3} dt + dB_t+dL_t^a ,
\end{equation}
for $0\leq{t} < 1$ with initial condition $X^a_0=0$, $L_a^0=0$, satisfying the conditions (a) $X^a_{t}\geq{0}$ for $0\leq{t}<1$, (b) $L_{t}^a$ is non-decreasing, and (c) $\int_{0}^{t}{\I{{{X^a_s>0}}}}{dL^a_{s}}=0$ for $0\leq{t}<1$. 

Indeed, let
\[
K_t^a = - \inf_{0 \le s \le t} \left[ a - a(1-s)^{2/3} +  \int_0^s \frac{d B_u}{(1-u)^{2/3}} \right].
\]
Then the solution to (\ref{eq sys sde omega}) is given for $0 \le t < 1$ by
\begin{align*}
L_t^a & = \int_0^t (1-s)^{2/3} dK_s^a, \\
X_t^a & = a(1-t)^{2/3} - a(1-t)^{4/3} + (1 - t)^{2/3}  \int_0^t \frac{d B_u}{(1-u)^{2/3}} + (1-t)^{2/3} K_t^a.
\end{align*}
Moreover, the following statements are true almost surely:
\[
X_1^a := \lim_{t\to{1-}}{X}_t^a = 0, \quad  \int_{0}^{1}\frac{X_s^a}{3(1-s)}ds<\infty, \quad L_1^a := \lim_{t \to 1-} L_t^a < \infty.
\]
Finally, for $a \ge 0$ and $r \ge a+1$ we have that 
\begin{equation*}
\sup_{0 \le t < 1} \Prob{\frac{X^{a}_t}{\sqrt{1-t}} > r} \le e^{-(r-a)^2/6}.
\end{equation*}
\end{prop}

\begin{proof}
Theorem 2.1.1 from \cite{Pilipenko} guarantees the existence of a unique solution to the reflected SDE on the time-interval $[0,t]$ for any fixed $t < 1$, since the diffusion and drift coefficients satisfy the appropriate Lipschitz and linear growth conditions. 

For $0\leq{t}<1$, define 
\begin{align*}
Y^a_t & = a - a(1-t)^{2/3} + \int_{0}^{t}\frac{d{B}_u}{(1-u)^{2/3}} - \inf_{0 \le s \le t} \left[ a - a(1-s)^{2/3} + \int_{0}^{s}\frac{d{B}_u}{(1-u)^{2/3}} \right] \\
 & = a - a(1-t)^{2/3} + \int_{0}^{t}\frac{d{B}_u}{(1-u)^{2/3}} + K_t^a.
\end{align*}
Then we straightforwardly have that
\[
dY^a_t = \frac{2a}{3}(1-t)^{-1/3}dt + (1-t)^{-2/3} dB_t + dK_t^a,
\]
and, by Skorokhod's lemma (see, for example, Lemma 2.1 of Chapter VI of \cite{RevuzYor}), $K^a$ is the local time at 0 of $Y^a$.  

Now let $X^a_t = (1-t)^{2/3} Y_t^a$ and $L^a_t = \int_0^t (1-s)^{2/3} dK_s^a$.  Then, by construction, $Y_t^a \ge 0$ for $0 \le t < 1$ and so $X_t^a \ge 0$ for $0 \le t < 1$.  Furthermore,
\begin{align*}
dX_t^a & = - \frac{2}{3} (1-t)^{-1/3} Y^a_t dt + (1-t)^{2/3} dY_t^a \\
& = -\frac{2}{3} \frac{X_t^a}{(1-t)} dt + \frac{2a}{3}(1-t)^{1/3}dt + dB_t + (1-t)^{2/3}dK_t^a \\
& = -\frac{2}{3} \frac{X_t^a}{(1-t)} dt + \frac{2a}{3}(1-t)^{1/3}dt + dB_t + dL_t^a.
\end{align*}
Clearly $L_0^a = 0$ and $L_t^a$ is non-decreasing.  Let us show that the local time at 0 of $X^a$ is $L^a$.  By Tanaka's formula, the local time is given by
\begin{align*}
& |X_t^a| - \int_0^t \mathrm{sgn}(X_s^a) dX_s^a \\
& = (1-t)^{2/3} |Y_t^a| + \frac{2}{3} \int_0^t \mathrm{sgn}(Y_s^a) Y_s^a (1-s)^{-1/3} ds - \int_0^t (1-s)^{2/3} \mathrm{sgn}(Y_s^a) dY_s^a \\
& = (1-t)^{2/3} |Y_t^a| + \frac{2}{3} \int_0^t |Y_s^a| (1-s)^{-1/3} ds - \int_0^t (1-s)^{2/3} \mathrm{sgn}(Y_s^a) dY_s^a \\
& = \int_0^t (1-s)^{2/3} d|Y_s^a| - \int_0^t (1-s)^{2/3} \mathrm{sgn}(Y_s^a) dY_s^a,
\end{align*}
by integration by parts.  But we also have $K_t^a = |Y_t^a| - \int_0^t \mathrm{sgn}(Y_s^a) dY_s^a$ and so the last line is equal to $\int_0^t (1-s)^{2/3} dK_s^a$, which is $L_t^a$.  It follows that $\int_0^t \I{X_s^a > 0} dL_s^a = 0$.  Hence, $(X^a, L^a)$ solves the reflected SDE.

For $a = 0$, we have
\[
X_t^0 = (1-t)^{2/3} \int_0^t \frac{dB_u}{(1-u)^{2/3}} - (1-t)^{2/3} \inf_{0 \le s \le t} \int_0^s \frac{dB_u}{(1-u)^{2/3}}
\]
and so
\[
(X_t^0, 0 \le t < 1) \equidist \left( (1-t)^{2/3} \left| \int_0^t \frac{dB_u}{(1-u)^{2/3}} \right|, 0 \le t < 1 \right).
\]
(See, for example, Exercise 1.3.1 of \cite{Pilipenko}.)

Let first show that $\lim_{t\to{1}}X^0_t=0$. By the Dubins-Schwarz theorem we have
\begin{equation} \label{eq Dubins Sshwarz appliacation}
\left({\int_{0}^{t}{\frac{dB_s}{(1-s)^{2/3}}}},0\leq{t}<1\right) \equidist \left({B_{f(t)}},0\leq{t}<1\right),
\end{equation}
where $f(t)=3(1-t)^{-1/3}-3$ is the quadratic variation of the process on the left-hand side of (\ref{eq Dubins Sshwarz appliacation}). Recall that
\[
(B_t, t \ge 0)  \equidist (t B_{{1}/{t}}, t \ge 0).
\]
Thus we have
\[
\left(X^0_t, 0\leq{t}<1\right) \equidist \left((1-t)^{2/3} f(t) \cdot |B_{1/f(t)}|, 0\leq{t}<1\right).
\]
But then
\[
(1-t)^{2/3}{f(t)}\cdot |B_{1/f(t)}|=3\left((1-t)^{1/3}-(1-t)^{2/3}\right) \left| B_{\frac{(1-t)^{1/3}}{3(1-(1-t)^{1/3})}} \right| \to{0}
\]
almost surely as $t\to{1}$.  It follows that $\lim_{t \to 1} X_t^0 = 0$ almost surely.  Moreover,
\[
\int_{0}^{1}{\frac{X_t^0}{3(1-t)}dt} \equidist \frac{1}{3}\int_{0}^{1}{(1-t)^{-1/3}\vert{B_{f(t)}}\vert}dt.
\]
The change of variables $u=f(t)$ yields that the right-hand side is equal to 
\[
9 \int_{0}^{1}{\frac{\vert{B_u}\vert}{(u+3)^{3}}}du.
\]
Since $B_u=_{d}{\sqrt{u}B_1}$ and $\E{\vert{B_u}\vert}=\sqrt{\frac{2u}{\pi}}$ we have
\[
\E{\int_{0}^{1}{\frac{X_t^0}{3(1-t)}dt}}=9{\sqrt{\frac{2}{\pi}}}{\int_{0}^{1}{\frac{\sqrt{u}}{(u+3)^{3}}}}du<\infty,
\]
and so
\[
 \int_{0}^{1}{\frac{X_t^0}{3(1-t)}dt} <\infty \text{ a.s.}
\]
Since we have
\[
X_t^0 = -2 \int_0^t \frac{X_s^0}{3(1-s)} ds + B_t + L_t^0,
\]
and $X_t^0$, $\int_{0}^{t}{\frac{X_s^0}{3(1-s)}dt}$ and $B_t$ all possess finite almost sure limits as $t \to 1$, the same must also be true of $L_t^0$.

To obtain the almost sure finiteness statements for a general $a \ge 0$, note that if we build $X^0$ and $X^a$ from the same Brownian motion then
\begin{equation} \label{eqn:bounds}
0 \le X_t^a \le a(1-t)^{2/3} + X_t^0.
\end{equation}
It follows that $\lim_{t \to 1} X_t^a = 0$ almost surely.  We also get
\[
\int_0^1 \frac{X^a_s}{3(1-s)} ds \le \frac{a}{2} + \int_0^1 \frac{X^0_s}{3(1-s)} ds < \infty
\]
almost surely.  Finally, by the same argument as in the $a=0$ case, we must then also have that $L_1^a := \lim_{t \to 1} L_t^a < \infty$ almost surely.

We now turn to the final statement. By (\ref{eqn:bounds}), it is sufficient to show the result for $a = 0$.  For any $0 \le t < 1$, we have
\begin{align*}
\Prob{\frac{X_t^0}{\sqrt{1-t}} > r} & = \Prob{(1-t)^{1/6} |B_{f(t)}| >r} \\
& = \Prob{|\mathrm{N}(0, 3 -3(1-t)^{1/3})| > r} \\
& \le \Prob{|\mathrm{N}(0, 1)| > r/\sqrt{3}} \le e^{-r^2/6},
\end{align*}
by standard Gaussian tail bounds.
\end{proof}

\subsection{The invariance principle} \label{Sect Invariance principle}

Theorem \ref{thm fluid limit for X for omega>>sqrt n} gives that for $\alpha(n)\gg\sqrt{n}$ the number of fires rescaled by $\alpha(n)$ possesses a fluid limit. This arises because the dominant contribution to the number of fires comes from connecting to nodes of degree $4$.  If we had no vertices of degree 4, we would connect only to vertices of degree 3, and the number of fires would make jumps of $+1$ or $-1$ only, each with probability $1/2$.  This suggests that $X^n$ should behave like a simple symmetric random walk.  If $\alpha(n) \sim a\sqrt{n}$, $a>0$ this random walk acquires a positive drift.

Recall that for $0\leq{i}\leq{\zeta^{n}}$, we have
\begin{equation*}
L^{n}(i)=2 \sum_{j=0}^{i-1}{\I{X^{n}(j)=0}}.
\end{equation*}
We will prove the following scaling limit.
\begin{thm} \label{Thm Convergence up to 1-epsilon}
Fix $t < 1$ and suppose $\alpha(n)/\sqrt{n}\to{a}$ for $a\geq{0}$. Then
\begin{equation*}
\left(\frac{X^{n}(\fl{ns})}{\sqrt{n}},\frac{L^{n}(\fl{ns})}{\sqrt{n}}, 0\le s \le t \right)\convdist{\left({X^{a}_s, L^{a}_s,0\le s \le t}\right)},
\end{equation*}
uniformly, where ${\left({X^{a}_s, L^{a}_s,0\le s \le 1}\right)}$ is the process from Proposition \ref{Prop sistem 1}.
\end{thm}

We will deduce Theorem~\ref{Thm Convergence up to 1-epsilon} from the following general invariance principle for reflected diffusions. Such invariance principles go back to Stroock and Varadhan~\cite{StroockVaradhan}, and we do not believe that this result is novel.  But since we could not find a version in the existing literature adapted to our particular setting, we give a proof in the appendix, using ideas from Ethier and Kurtz~\cite{EthierKurtz} and Kang and Williams~\cite{KangWilliams}.

\begin{thm} \label{Theorem Invariance principle}
Let $q: \R_+ \times \R_+ \to \R_+$, $b: \R_+ \times \R_+ \to \R$ satisfy the following conditions:
\begin{itemize}
\item[(1)] (Lipschitz condition)
\begin{equation*}
\exists{K>0}\quad \forall{t \ge 0}\quad \forall{y_{1},y_{2}\in{\R_+}}: \vert{q(t,y_{1})-q(t,y_{2})}\vert+\vert{b(t,y_{1})-b(t,y_{2})}\vert\leq{K\vert{y_{1}-y_{2}}\vert};
\end{equation*}
\item[(2)] (linear growth condition)
\begin{equation*}
\exists{C>0} \quad \forall{t\geq{0}} \quad \forall{x\in{\R_+}}: \vert{q(t,y)}\vert+\vert{b(t,y)}\vert\leq{C(1+\vert{y}\vert)}.
\end{equation*}
\end{itemize}
Let  $Y^n, B^n$ be $\mathbb{D}(\R_+, \R)$-valued stochastic processes and let $Q^n, L^n$ be increasing $\mathbb{D}(\R_+, \R)$-valued stochastic processes.  Let $\mathcal{F}_t^n = \sigma(Y^n(s), L^n(s), Q^n(s), B^n(s): s \le t)$ and let $\tau^n(r) = \inf\{t \ge 0: |Y^n(t)| \ge r \text{ or } |Y^n(t-)| \ge r\}$.
Suppose that
\begin{enumerate}
\item[(a)] $M^n = Y^n - B^n$ is an $(\mathcal{F}^n_t)$-local martingale,
\item[(b)] $(M^n)^2 - Q^n$ is an $(\mathcal{F}^n_t)$-local martingale,
\end{enumerate}
and that for all $r > 0$ and $T > 0$,
\begin{enumerate}
\item[(c)] $\lim_{n \to \infty} \E{\sup_{t \le T \wedge \tau^n(r)} | Y^n(t) - Y^n(t-)|^2} = 0$,
\item[(d)] $\lim_{n \to \infty} \E{\sup_{t \le T \wedge \tau^n(r)} | B^n(t) - B^n(t-)|^2} = 0$,
\item[(e)] $\lim_{n \to \infty} \E{\sup_{t \le T \wedge \tau^n(r)} | Q^n(t) - Q^n(t-)|} = 0$,
\item[(f)] $\sup_{t \le T \wedge \tau^n(r)} \left| B^n(t) - L^n(t) - \int_0^t b(s,Y^n(s)) ds \right| \convprob 0$ as $n \to \infty$,
\item[(g)] $\sup_{t \le T \wedge \tau^n(r)} \left| Q^n(t) - \int_0^t q(s,Y^n(s)) ds \right| \convprob 0$ as $n \to \infty$,
\item[(h)] $Y^n(0) \convprob 0$ as $n \to \infty$,
\item[(i)] there exists a sequence $(\delta_n)_{n \ge 1}$ of constants with $\delta_n \to 0$ such that $L^n(0) = 0$, $L^n(t) - L^n(t-) \le \delta_n$ and $L^n(t) = \int_0^t \I{Y^n(s) \le \delta_n} d L^n(s)$.
\end{enumerate}
Then $(Y^n, L^n) \convdist (Y,L)$ in the Skorokhod topology as $n \to \infty$, where $(Y,L)$ is the unique solution to the reflected SDE specified by $Y_0 = L_0 = 0$, $L$ is non-decreasing, and
\[
dY_t = b(t,Y_t) dt + \sigma(t, Y_t) dW_t + dL_t,
\]
where $\int_0^t \I{Y_s > 0} dL_s=0$, $W$ is a standard Brownian motion and $\sigma(t,y) = \sqrt{q(t,y)}$.
\end{thm}

\begin{proof}[Proof of Theorem \ref{Thm Convergence up to 1-epsilon}]
Define
\[
\tilde{X}^{n}(s)=\frac{{X}^{n}(\lfloor{ns}\rfloor)}{\sqrt{n}} \quad \text{and} \quad \tilde{L}^{n}(s) =\frac{{L}^{n}(\lfloor{ns}\rfloor)}{\sqrt{n}}.
\]
Recall from Proposition \ref{Prop sistem 1} that ${\left({X^{a}_s, L^{a}_s,0\le s \le t}\right)}$ is the solution to the reflected SDE
\begin{equation*} 
dX_s^a=-\frac{2}{3}{\frac{X_s^a}{(1-s)}}ds+\frac{2a}{3}(1 - s)^{1/3} ds + dB_s+dL_s^a.
\end{equation*}
For $0\leq s \leq t$ and $x\geq{0}$, let $q(s,x)=1$ and $b(s,x)=-\frac{2}{3}{\frac{x}{(1-s)}}+\frac{2a}{3}(1-s)^{1/3}$. It follows straightforwardly that for $0\leq s \leq t$, the functions $q$ and $b$ satisfy the Lipschitz and linear growth conditions (1) and (2) in Theorem \ref{Theorem Invariance principle}. 

For $i\leq{\fl{nt}}$, let
\[
{B}^{n}(i) = \sum_{j=0}^{i-1} \E{X^n(j+1) - X^n(j) | \mathcal{F}^n_j} ,
\]
and
$$M^n_{X}(i) := X^n(i) - B^n(i),$$
as well as
\begin{align*}
Q^{n}(i) & = \sum_{j=0}^{i-1} \E{M_X^n(j+1)^2 - M_X^n(j)^2 | \mathcal{F}^n_j}  = \sum_{j=0}^{i-1}  \E{({{M}_X^{n}(j+1)}-M_X^n(j))^{2}\vert{\mathcal{F}^{n}_{j}}} \\
&= \sum_{j=0}^{i-1} \E{({X}^{n}(j+1)-{{X}^{n}(j)})^{2}\vert{{\mathcal{F}^{n}_{i}}}}-\E{{X}^{n}(j+1)-X^{n}(j)\vert{{\mathcal{F}^{n}_{j}}}}^{2}.
\end{align*}
Since everything is bounded, $M_X^n$ and $(M_X^n)^2 - Q^n$ are both $(\mathcal{F}_i^n)$-martingales.  Let 
\[
\tilde{M}(s)=\frac{M_X^n(\fl{ns})}{\sqrt{n}}, \quad \tilde{B}^n(s) = \frac{B^n(\fl{ns})}{\sqrt{n}} \quad \text{and} \quad \tilde{Q}^{n}(s)=\frac{Q^{n}(\fl{ns})}{{n}}
\]
be the appropriately rescaled versions of these quantities.  Then (a) and (b) hold by construction and the main task facing us is to check that conditions (f) and (g) of the Theorem are fulfilled for $\tilde{M}^n$ and $(\tilde{M}^n)^2 - \tilde{Q}^n$.

By Lemma~\ref{Lem expected increments of X},
\begin{equation*}
\begin{aligned}
B^n(i) & =2 \sum_{j=0}^{i-1} \I{X^n(j)=0} + \sum_{j=0}^{i-1} \frac{2\alpha(n)(1-j/n)^{4/3}-2X^n(j)}{3(n-j)}+ \sum_{j=0}^{i-1} E_{X,1}^{n}(j) \\
& = L^n(i) + \sum_{j=0}^{i-1} \frac{2\alpha(n)(1-j/n)^{4/3}-2X^n(j)}{3(n-j)}+ \sum_{j=0}^{i-1} E_{X,1}^{n}(j). 
\end{aligned}
\end{equation*}
Then
\begin{align*}
& \sup_{0\leq s \le t}\left| \tilde{B}^{n}(s)-\tilde{L}^{n}(s)-\int_{0}^{s}\left(\frac{2a}{3}(1-u)^{1/3} - \frac{2\tilde{X}^{n}(u)}{(1-u)} \right)du \right|  \\
& \le \sup_{0\le s \le t}\left| \sum_{j=0}^{\fl{ns} -1} \frac{1}{n} \left( \frac{2\alpha(n)}{3\sqrt{n}}(1 - j/n)^{1/3} - \frac{2\tilde{X}^{n}(j/n)}{3(1-j/n)}\right) -\int_{0}^{s}\left( \frac{2a}{3}(1-u)^{1/3}-\frac{2\tilde{X}^{n}(u)}{3(1-u)} \right) du\right| \\
& \qquad + \sup_{0\le i \le \fl{nt}}\sqrt{n} |E_{X,1}^{n}(i)|,
\end{align*}
which converges to 0 in probability as $n \to \infty$ by Lemma~\ref{Lem expected increments of X}.  Hence, condition (f) holds.

By Lemma~\ref{Lem expected increments of X},
\begin{align*}
Q^n(i) & = \sum_{j=0}^{i-1} \left(1+E_{X,2}^{n}(j) -4\I{X^{n}(j)=0} \left(\frac{2\alpha(n)(1-j/n)^{4/3}-2X^{n}(j)}{3(n-j)}+E_{X,1}^{n}(j) \right) \right. \\
& \qquad \qquad \left. -\left(\frac{2\alpha(n)(1-j/n)^{4/3}-2X^{n}(j)}{3(n-j)}+E_{X,1}^{n}(j) \right)^{2} \right) \\
& = i + \sum_{j=0}^{i-1}\left(E_{X,2}^{n}(j)  -4\I{X^{n}(j)=0} \left( \frac{2\alpha(n)(1-j/n)^{4/3}-2X^{n}(j)}{3(n-j)}+E_{X,1}^{n}(j) \right) \right. \\
& \qquad \qquad \left. -\left(\frac{2\alpha(n)(1-j/n)^{4/3}-2X^{n}(j)}{3(n-j)} + E_{X,1}^{n}(j) \right)^{2} \right).
\end{align*}
Then
\begin{align*}
&\sup_{0\leq s \leq t} \left| \tilde{Q}^{n}(s)-s \right| \\
&=\sup_{0\leq s \leq t} \Bigg| \frac{\fl{ns}}{n}-s+\sum_{j=0}^{\fl{ns}-1} \Bigg( \frac{E_{X,2}^{n}(j)}{n}
- \left(\frac{2(\alpha(n)/\sqrt{n})(1-j/n)^{4/3}-2\tilde{X}(j/n)}{3(n-j)}+\frac{E_{X,1}^{n}(j)}{\sqrt{n}}\right)^{2} \\
& \qquad \qquad \qquad \qquad  -\frac{4}{\sqrt{n}} \I{\tilde{X}^{n}(j/n)=0} \left( \frac{2(\alpha(n)/\sqrt{n})(1-j/n)^{4/3}-2\tilde{X}(j/n)}{3(n-j)}+\frac{E_{X,1}^{n}(j)}{\sqrt{n}}\right) \Bigg)\Bigg|.
\end{align*}
Now, for large enough $n$ we have
\[
\begin{aligned}
& \left| \frac{2(\alpha(n)/\sqrt{n})(1-j/n)^{4/3}-2\tilde{X}(j/n)}{3(n-j)}+\frac{E_{X,1}^{n}(j)}{\sqrt{n}} \right| \\
& \qquad  \le \frac{1}{n} \frac{4a}{3(1-t)^{1/3}} + \frac{1}{n} \frac{2}{3(1-t)} \sup_{0 \le s \le t} \tilde{X}^n(s) + \frac{\sup_{0 \le i \le \fl{nt}} |E_{X,1}^n(i)|}{\sqrt{n}}.
\end{aligned}
\]
Since $\sup_{0\leq{i}\leq{\lfloor{nt}\rfloor}}{|E_{X,2}^{n}(i)|}\convprob{0}$, $\sup_{0 \le i \le \fl{nt}} |E_{X,1}^n(i)| \convprob 0$ and $\sup_{0 \le s \le t} \tilde{X}^n(s)$ is tight, it follows that
\[
\sup_{0\leq s \leq t} \left| \tilde{Q}^{n}(s)-s \right| \convprob 0
\]
as $n \to \infty$, and so we have that condition (g) of Theorem \ref{Theorem Invariance principle} is fulfilled.

Since the increments of $X^n$ are bounded, conditions (c), (d) and (e) of Theorem \ref{Theorem Invariance principle} are trivially satisfied. As $\tilde{X}^{n}(0)=0$, condition (h) is satisfied. Finally, note that condition (i) is fulfilled if we take $\delta_n = \frac{2}{\sqrt{n}}$, since $\tilde{L}^n(s) - \tilde{L}^n(s-) \le \frac{2}{\sqrt{n}}$ and
\[
\int_0^s \I{\tilde{X}^n(u) \le \frac{2}{\sqrt{n}}} dL^n(u) = 2 \sum_{i=0}^{\fl{ns}-1} \I{\tilde{X}^n(i) = 0} = L^n(s).
\]
Applying Theorem \ref{Theorem Invariance principle} completes the proof.
 \end{proof}

Finally, we are ready to prove the convergence of the suitably rescaled number of clashes we encounter up to a fixed time $t < 1$.

\begin{lem} \label{Lem X N joint convergence up to 1-epsilon}
For fixed $0 < t < 1$ we have
\[
\left(\frac{N^{n}(\lfloor{ns}\rfloor)}{\sqrt{n}}, 0\leq s \leq t \right) \convdist  \left(N_s^a,  0\leq s \leq t
\right)\]
uniformly as $n \to \infty$.
\end{lem}

\begin{proof}
The argument is standard, since $N^n$ is a counting process which is bounded by $(3n+4\alpha(n))/2$.  We have that
\[
M_N^n(k) = N^n(k) - \sum_{i=1}^{k-1} \E{N^n(i+1) - N^n(i) | \mathcal{F}_i^n}
\]
defines a martingale.  We also have
\begin{align*}
\E{M^n_N(k)^2} & = \E{\sum_{i=0}^{k-1} \mathrm{var} \left( N^n(i+1) - N^n(i) | \mathcal{F}_i^n \right) } \\
& \le \E{\sum_{i=0}^{k-1} \E{N^n(i+1) - N^n(i) | \mathcal{F}_i^n}} = \E{N^n(k)},
\end{align*}
where the inequality holds since $N^n(i+1) - N^n(i) \in \{0,1\}$.  Now,
\begin{equation} \label{eqn:NnUI}
\begin{aligned}
\frac{1}{\sqrt{n}} \E{N^n(\fl{nt})} & \le \frac{1}{\sqrt{n}} \sum_{i=1}^{\fl{nt}-1} \E{ \frac{3n + 2 X^n(i) (3n + 4 \alpha(n) + 2 X^n(i))}{2(3n(1-t) - 1) (3n(1-t) - 3)}} \\
& \le \frac{\sqrt{n} t \left(3n + 2 \E{\sup_{0 \le i \le \zeta_n} X^n(i)} (3n + 4 \alpha(n) + 2n + 2 \alpha(n))\right)}{2(3n(1-t) -1) (3n(1-t) - 3)} \le C
\end{aligned}
\end{equation}
for some constant $C$, uniformly in $n$.  In particular, $\E{N^n(\fl{nt})}/n \to 0$. Fix $\delta > 0$.  Using Markov's inequality, we obtain that
\[
\frac{M_N^n(\fl{nt})}{\sqrt{n}} \convprob 0,
\]
as $n \to \infty$.

Now, by Lemma~\ref{Lem Exp increments N} we have
\[
\sqrt{n} \sup_{0 \le i \le \fl{nt}} \left| \E{N^n(i+1) - N^n(i) | \mathcal{F}^n_i} - \frac{X^n(i)}{3(n-i)} \right|
\convprob 0
\]
and so
\[
\sup_{0\leq s \leq t}\frac{1}{\sqrt{n}} \left| \sum_{i=1}^{\fl{ns}-1} \E{N^n(i+1) - N^n(i) | \mathcal{F}_i^n} - \sum_{i=0}^{\fl{ns} - 1} \frac{X^n(i)}{3(n-i)} \right| \convprob 0.
\]
Finally, we have
\[
\frac{1}{\sqrt{n}} \sum_{i=0}^{\fl{ns} - 1} \frac{X^n(i)}{3(n-i)} = \frac{1}{\sqrt{n}} \int_0^{\fl{ns}/n} \frac{X^n(\fl{nu})}{3\left(1 - \frac{\fl{nu}}{n}\right)} du.
\]
But then using Theorem \ref{Thm Convergence up to 1-epsilon} and the continuous mapping theorem, we get
\[
 \frac{1}{\sqrt{n}} \int_0^{\fl{ns}/n} \frac{X^n(\fl{nu})} {3\left(1 - \frac{\fl{nu}}{n}\right)} du \convdist \int_0^{s} \frac{X^a_u}{3 (1-u)} du = N^a_s,
 \]
 uniformly for $0\leq s \leq t$.
 Hence, as $n \to \infty$
 \[
\frac{ N^n(\fl{ns})}{\sqrt{n}} \convdist N_s^a
 \]
uniformly for $0\leq s \le t$, as desired.
\end{proof}

%
%
\section{The end of the process} \label{sec:end}

We now understand the scaling behaviour of the process $(X^n(i), L^n(i), N^n(i))_{i \ge 0}$ on the time-interval $[0,\fl{n(1-\epsilon)}]$ for any $\epsilon > 0$.  It remains to get from $\fl{n(1-\epsilon)}$ to $\zeta_n$.

We first show that $V^n(n)$ is small and that $X^n(\fl{n(1-\epsilon)}), L^n(\fl{n(1-\epsilon)}), N^n(\fl{n(1-\epsilon)})$ are,  for small $\epsilon$, good approximations of $X^n(n), L^n(n), N^n(n)$, by careful use of the coupling from Section~\ref{sec:coupling}.

\begin{prop} \label{prop:1-epsto1}
For any $\delta > 0$,
\begin{enumerate}
\item[(a)] if $\alpha(n) \to \infty$ then $V^n(n)/\alpha(n) \convprob 0$ as $n\to \infty$,
\item[(b)] $\lim_{\epsilon \to 0} \limsup_{n \to \infty} \Prob{ \frac{1}{\alpha(n) \vee \sqrt{n}} \left( N^n(n) - N^n(\fl{n(1 - \epsilon)}) \right) > \delta} = 0,$
\item[(c)] $\lim_{\epsilon \to 0} \limsup_{n \to \infty} \Prob{ \frac{1}{\alpha(n) \vee \sqrt{n}} \sup_{\fl{n(1-\epsilon)} \le i \le n} X^n(i) > \delta} = 0$,
\item[(d)] $\lim_{\epsilon \to 0} \limsup_{n \to \infty} \Prob{ \frac{1}{\alpha(n) \vee \sqrt{n}} \left( L^n(n) - L^n(\fl{n(1 - \epsilon)}) \right) > \delta} = 0.$
\end{enumerate}
\end{prop}

\begin{proof}
(a) We have
\[
\E{V^n(\fl{n(1-\epsilon)})} \le 2 \epsilon^{4/3} \alpha(n) + \alpha(n) \Prob{V^n(\fl{(1-\epsilon)n}) > 2 \epsilon^{4/3} \alpha(n)}
\]
 and, by Lemma~\ref{Lem V3 fluid limit}, $\Prob{V^n(\fl{(1-\epsilon)n}) > 2 \epsilon^{4/3}\alpha(n)} \to 0$ as $n \to \infty$.  But $V^n$ is decreasing and so
\begin{equation} \label{eqn:Vbound}
\limsup_{n \to \infty} \frac{\E{V^n(n)}}{\alpha(n)} \le \limsup_{n \to \infty} \frac{\E{V^n(\fl{n(1-\epsilon)})}}{\alpha(n)} \le 2\epsilon^{4/3}.
\end{equation}
Since $V^n(n) \ge 0$ and this inequality holds for any $\epsilon > 0$, we must have that the left-hand side is, in fact, equal to 0. Markov's inequality then gives that $V^n(n)/\alpha(n) \convprob 0$.

(b) Recall the coupling and notation from Section~\ref{sec:coupling}.  First note that
\begin{equation} \label{eqn:N=N_1+N_2}
\begin{aligned}
\E{N^n(n) - N^n(\fl{n(1-\epsilon)})} & = \E{N_1^n(n) - N_1^n(\fl{n(1-\epsilon)})} + \E{N_2^n(n) - N_2^n(\fl{n(1-\epsilon)})} \\
& \le \E{N_1^n(n) - N_1^n(\fl{n(1-\epsilon)})} + \E{V^n(\fl{n(1-\epsilon)})}.
\end{aligned}
\end{equation}
From (\ref{eqn:N_1}), the process
\[
\left( N_1^n(k) - \sum_{i=0}^{k-1} \tfrac{3U^n(i) + 2(X_1^n(i) - \I{X_1^n(i) > 0})(3U^n(i) + 4V^n(i) + X^n(i) - 3)}{2(3U^n(i) + 4V^n(i) + X^n(i) -\I{X^n(i) > 0})(3U^n(i) + 4V^n(i) + X^n(i) -\I{X^n(i) > 0} - 2)}, k \ge 0 \right)
\]
is a martingale.  Hence,
\begin{align}
& \E{N_1^n(n) - N_1^n(\fl{n(1-\epsilon)})}  \notag \\
& \qquad = \E{\sum_{i=\fl{n(1-\epsilon)}}^{n-1} \tfrac{3U^n(i) + 2(X_1^n(i) - \I{X_1^n(i) > 0})(3U^n(i) + 4V^n(i) + X^n(i) - 3)}{2(3U^n(i) + 4V^n(i) + X^n(i) -\I{X^n(i) > 0})(3U^n(i) + 4V^n(i) + X^n(i) -\I{X^n(i) > 0} - 2)}} \notag \\
& \qquad \le \frac{1}{6} \sum_{i=\fl{n(1-\epsilon)}}^{n-1} \frac{1}{(n-i)} + \E{\sum_{i=\fl{n(1-\epsilon)}}^{n-1} \frac{X_1^n(i) - \I{X_1^n(i) > 0}}{3U^n(i) +  4V^n(i) + X^n(i) -\I{X^n(i) > 0}}} \notag \\
& \qquad \le \frac{1}{6}(\log(\ce{n\epsilon}) + 1) + \E{\sum_{i=\fl{n(1-\epsilon)}}^{n-1} \frac{X_1^n(i) - \I{X_1^n(i) > 0}}{3U^n(i) + 4 V^n(i) + X^n(i) -\I{X^n(i) > 0}}}. \label{eqn:nbound}
\end{align}
We fix $C > \sqrt{6}$ and define a sequence of stopping times.  Let
\[
T_0 = \inf\{k \ge \fl{n(1-\epsilon)}: X_1^n(k) \ge \fl{C \sqrt{n-k}}\},
\]
and, inductively, for $i \ge 1$, let
\[
\begin{aligned}
S_i & = \inf\{k \ge T_{i-1}: X_1^n(k) = 0 \} \wedge n \\
T_i & = \inf\{k \ge S_i: X_1^n(k) = \fl{C \sqrt{n-k}}\} \wedge n.
\end{aligned}
\]
Then since $U^n(i) \ge n-i$,
\[
\begin{aligned}
& \E{\sum_{i=\fl{n(1-\epsilon)}}^{n-1} \frac{X_1^n(i) - \I{X_1^n(i) > 0}}{3U^n(i) + 4 V^n(i) +  X^n(i) -\I{X^n(i) > 0}}} \\
& \le \E{\sum_{i=\fl{n(1-\epsilon)}}^{T_0-1} \frac{X_1^n(i) - \I{X_1^n(i) > 0}}{3(n-i)}}
+ \E{\sum_{k=1}^{n} \sum_{i=T_{k-1}}^{S_k-1} \frac{X_1^n(i) - \I{X_1^n(i) > 0}}{3U^n(i) + 4V^n(i) + X^n(i) -\I{X^n(i) > 0}}} \\
& \qquad + \E{\sum_{k=1}^{n} \sum_{i=S_k}^{T_k-1} \frac{X_1^n(i) - \I{X_1^n(i) > 0}}{3(n-i)}}.
\end{aligned}
\]
Either $T_0 = \fl{n(1-\epsilon)}$ (in which case the first sum is empty) or on the time interval $[\fl{n(1-\epsilon)}, T_0 - 1]$ we have $X_1^n(i) -1 \le C \sqrt{n-i}$.  Likewise, on $[S_k, T_k-1]$, we have $X_1^n(i) -1 \le C \sqrt{n-i}$.  So
\[
\E{\sum_{i=\fl{n(1-\epsilon)}}^{T_0-1} \frac{X_1^n(i) - \I{X_1^n(i) > 0}}{3(n-i)}}+ \E{\sum_{k=1}^{n} \sum_{i=S_k}^{T_k-1} \frac{X_1^n(i) - \I{X_1^n(i) > 0}}{3(n-i)}} \le \frac{C}{3} \sum_{i=\fl{n(1-\epsilon)}}^{n-1} \frac{1}{\sqrt{n-i}}.
\]
Since
\[
\sum_{i=1}^{\ce{n\epsilon}} \frac{1}{\sqrt{i}} \le \int_0^{\ce{n\epsilon}} \frac{1}{\sqrt{x}} dx \le 2 \sqrt{n\epsilon +1},
\]
we get
\begin{equation} \label{eqn:rootn}
\E{\sum_{i=\fl{n(1-\epsilon)}}^{T_0-1} \frac{X_1^n(i) - \I{X_1^n(i) > 0}}{3(n-i)}}+ \E{\sum_{k=1}^{n} \sum_{i=S_k}^{T_k-1} \frac{X_1^n(i) - \I{X_1^n(i) > 0}}{3(n-i)}} \le \frac{2C}{3}\sqrt{n \epsilon +1}.
\end{equation}
It therefore remains to deal with the term
\[
\E{\sum_{k=1}^{n} \sum_{i=T_{k-1}}^{S_k-1} \frac{X_1^n(i) - \I{X_1^n(i) > 0}}{3U^n(i) + 4V^n(i) + X^n(i) -\I{X^n(i) > 0}}}.
\]

Now note from (\ref{eqn:X_1}) that 
\begin{equation*} 
\begin{aligned}
& \left(X_1^n(k) + \sum_{i=0}^{k-1}\frac{2X_1^n(i) - 2\I{X_1^n(i) > 0} + X^n_2(i) - \I{X_1^n(i) = 0, X_2^n(i) > 0}}{3U^n(i) + 4 V^n(i) +  X^n(i) -  \I{X^n(i) > 0}} \right. \\
& \left. \qquad \qquad \qquad - \sum_{i=0}^{k-1} \I{X_1^n(i) = 0, X_2^n(i) > 0} - 2\sum_{i=0}^{k-1} \I{X_1^n(i) = X_2^n(i) = 0} , k \ge 0\right)
\end{aligned}
\end{equation*}
is a martingale.  Since $X_1^n$ does not touch 0 on the time interval $[T_{k-1}, S_k - 1]$, we have by the optional stopping theorem that
\[
\begin{aligned}
& \E{\sum_{i=T_{k-1}}^{S_k-1} \frac{X_1^n(i) - \I{X_1^n(i) > 0}}{3U^n(i) + 4V^n(i) +  X^n(i) -\I{X^n(i) > 0}}} \\
&\qquad  = \frac{1}{2} \E{X_1^n(T_{k-1}) - X_1^n(S_{k})} - \frac{1}{2}\E{\sum_{i=T_{k-1}}^{S_k-1} \frac{X^n_2(i) -\I{X_1^n(i) = 0, X_2^n(i) > 0}}{3U^n(i) + 4 V^n(i) +  X^n(i) -  \I{X^n(i) > 0}} }\\
& \qquad \le \frac{1}{2} \E{X_1^n(T_{k-1})} \\
& \qquad \le \frac{C \E{\sqrt{n-T_{k-1}}}}{2} \\
&\qquad \le \frac{C \sqrt{n - \E{T_{k-1}}}}{2},
\end{aligned}
\]
where the last line follows using Jensen's inequality.  So we have
\begin{equation} \label{eqn:sum}
\begin{aligned}
& \E{\sum_{k=1}^{n} \sum_{i=T_{k-1}}^{S_k-1} \frac{X_1^n(i) - \I{X_1^n(i) > 0}}{3U^n(i) + 4V^n(i) + X^n(i) -\I{X^n(i) > 0}}} \le \frac{C}{2} \sum_{k=0}^n \sqrt{n - \E{T_{k}}}.
\end{aligned}
\end{equation}

We thus need a lower bound on $\E{T_k}$ for any $k$.  Let us assume that at some time $\fl{n(1-\epsilon)} \le \ell < n$, we have $X_1^n(\ell) = 0$.  In order to find such a lower bound, we use the coupling from Section~\ref{sec:coupling}, which yields a SSRW reflected above 2, $Y^n$, such that $Y^n(\ell) = 2$ and $X_1^n(i) \le Y^n(i)$ for all $i \ge \ell$.  Recall that
\[
(Y^n(i))_{i \ge \ell} \equidist (2 + |Z(i)|)_{i \ge \ell},
\]
where $Z$ is a SSRW with $Z(\ell) = 0$.  Then if $\sigma = \inf\{i \ge \ell: 2+|Z(i)| = \fl{C \sqrt{n-i}}\}$, we have that $\sigma$ is stochastically smaller than $T_k$ conditioned on $S_k = \ell$.  Now, $(Z(i)^2 - i)_{i \ge \ell}$ is a martingale, and so
\[
\E{Z(\sigma)^2 - \sigma} = -\ell.
\]
But
\[
\E{Z(\sigma)^2} = \E{(\fl{C\sqrt{n-\sigma}} - 2)^2} \ge \E{(C \sqrt{n-\sigma} - 3)^2} \ge C^2 (n - \E{\sigma}) - 6 \E{\sqrt{n - \sigma}}.
\]
Since $n - \sigma \in \Z$, we have the very crude bound $\sqrt{n-\sigma} \le n-\sigma$.  Hence,
\[
\E{\sigma} - \ell = \E{Z(\sigma)^2} \ge (C^2-6)(n - \E{\sigma}).
\]
Then
\[
\E{\sigma} \ge \frac{(C^2-6)n + \ell}{C^2 - 5}.
\]
We obtain, by the stochastic domination, that for $k \ge 1$,
\[
\E{T_k} \ge \frac{(C^2 - 6)n + \E{S_k}}{C^2 - 5}.
\]
Since $S_k \ge T_{k-1}$, we get
\[
\E{T_k} \ge \frac{(C^2 - 6)n + \E{T_{k-1}}}{C^2 - 5},
\]
and so
\[
n - \E{T_k} \le \frac{n-\E{T_{k-1}}}{C^2 - 5}.
\]
By induction,
\[
n - \E{T_k} \le \frac{n - \E{T_0}}{(C^2 - 5)^k} \le \frac{\ce{n\epsilon}}{(C^2 - 5)^k},
\]
since $T_0 \ge \fl{n(1-\epsilon)}$.  It follows from (\ref{eqn:sum}) that
\[
\begin{aligned}
\E{\sum_{k=1}^{n} \sum_{i=T_{k-1}}^{S_k-1} \frac{X_1^n(i) - \I{X_1^n(i) > 0}}{3U^n(i) + 4V^n(i) + X^n(i) -\I{X^n(i) > 0}}} 
& \le \frac{C}{2} \sum_{k=0}^{\infty} \frac{\sqrt{n \epsilon + 1}}{(C^2 - 5)^{k/2}} \\
& = \frac{C\sqrt{C^2 - 5}}{2(\sqrt{C^2 - 5} - 1)} \sqrt{n \epsilon + 1}.
\end{aligned}
\]
Putting this together with (\ref{eqn:N=N_1+N_2}), (\ref{eqn:nbound}) and (\ref{eqn:rootn}) we get
\[
\begin{aligned}
& \E{N^n(n) - N^n(\fl{n(1-\epsilon)})} \\
 & \qquad \le \frac{1}{6}(\log(\ce{n\epsilon}) + 1) + C \sqrt{n \epsilon+1} + \frac{C\sqrt{C^2 - 5}}{2(\sqrt{C^2 - 5} - 1)} \sqrt{n \epsilon + 1} + \E{V^n(\fl{(1-\epsilon)n})}.
 \end{aligned}
\]
Using (\ref{eqn:Vbound}), we have
\[
\limsup_{n \to \infty} \frac{\E{V^n(\fl{n(1-\epsilon)})}}{\alpha(n)} \le 2\epsilon^{4/3}.
\] 
 Hence,
\begin{equation} \label{eqn:Nnbound}
\limsup_{n \to \infty} \frac{1}{\alpha(n) \vee \sqrt{n}} \E{N^n(n) - N^n(\fl{n(1-\epsilon)})} \le \left(C + \frac{C\sqrt{C^2 - 5}}{2(\sqrt{C^2 - 5} - 1)} \right) \sqrt{\epsilon} + 2 \epsilon^{4/3}.
\end{equation}
Applying Markov's inequality and taking $\epsilon \to 0$ then yields that for $\delta > 0$,
\[
\lim_{\epsilon \to 0} \limsup_{n \to \infty} \Prob{ \frac{1}{\alpha(n) \vee \sqrt{n}} \left( N^n(n) - N^n(\fl{n(1 - \epsilon)}) \right) > \delta} = 0.
\]

(c) We will show that $\frac{1}{\alpha(n) \vee \sqrt{n}} \E{\sup_{\fl{n(1-\epsilon)} \le i \le n} X^n(i)}$ is small.  We again use the coupling of $X^n_1$ with $Y^n$, but started now with $Y^n(\fl{n(1-\epsilon)}) = X_1^n(\fl{n(1-\epsilon)})+2$. This gives
\[
(Y^n(\fl{n(1-\epsilon)}+k))_{k \ge 0} \equidist (2 + |X_1^n(\fl{(1-\epsilon)n}) + Z(k)|)_{k \ge 0}.
\]
Since $X^n(i) = X_1^n(i) + X_2^n(i)$ with $X_2^n(i) \le V^n(\fl{n(1-\epsilon)})$ for $i \ge \fl{n(1-\epsilon)}$, we have
\[
\sup_{\fl{n(1-\epsilon)} \le i \le n} X^n(i) \le V^n(\fl{n(1-\epsilon)}) + \sup_{\fl{n(1-\epsilon)} \le i \le n} Y^n(i)
\]
and so we get
\begin{align*}
& \frac{1}{\alpha(n) \vee \sqrt{n}}\E{\sup_{\fl{n(1-\epsilon)} \le i \le n} X^n(i)} \\
& \qquad \le  \frac{2+ \E{V^n(\fl{n(1-\epsilon)})} + \E{X_1^n(\fl{n(1-\epsilon)})}}{\alpha(n) \vee \sqrt{n}} + \frac{1}{\sqrt{n}} \E{\sup_{0 \le i \le \ce{n\epsilon}} |Z(i)|}  \\
& \qquad \le \frac{2+\E{V^n(\fl{n(1-\epsilon)})}+ \E{X_1^n(\fl{n(1-\epsilon)})} }{\alpha(n)\vee \sqrt{n}} + 2 \sqrt{\epsilon},
\end{align*}
by Doob's $L^2$ inequality. We have already shown above that
\[
\lim_{\epsilon \to 0} \limsup_{n \to \infty} \frac{\E{V^n(\fl{n(1-\epsilon)})}}{\alpha(n) \vee \sqrt{n}} = 0.
\]
For $\alpha(n)/\sqrt{n} \to a$, we have
\[
\frac{X^n(\fl{n(1-\epsilon)})}{\sqrt{n}} \convdist X^a_{1-\epsilon}
\]
as $n \to \infty$.  By Lemma~\ref{Lem X is little o of n}, $X^n(\fl{n(1-\epsilon)})/\sqrt{n}$ is a uniformly integrable sequence of random variables, and so we obtain
\[
\E{\frac{X^n(\fl{n(1-\epsilon)})}{\sqrt{n}}} \to \E{X^a_{1-\epsilon}}
\]
as $n \to \infty$. Hence,
\[
\limsup_{n \to \infty} \frac{\E{\sup_{\fl{n(1-\epsilon)} \le i \le n} X^n(i)}}{\sqrt{n}} \le 2a\epsilon^{4/3} + \E{X_{1-\epsilon}^a} + 2\sqrt{\epsilon}.
\]
By the final statement of Proposition~\ref{Prop sistem 1}, we have
\[
\Prob{X_{1-\epsilon}^a > r \sqrt{\epsilon}} \le e^{-(r-a)^2/6}
\]
for $r > a$ and it follows, in particular, that $\E{X_{1-\epsilon}^a} \le C \sqrt{\epsilon}$ for some constant $C>0$.  

For $\alpha(n) > > \sqrt{n}$, we have $X^n(\fl{n(1-\epsilon)})/\alpha(n) \convprob \epsilon^{2/3} - \epsilon^{4/3}$.  Again, using the uniform integrability from  Lemma~\ref{Lem X is little o of n}, we get $\E{X^n(\fl{n(1-\epsilon)})}/\alpha(n) \convprob \epsilon^{2/3} - \epsilon^{4/3}$ and so
\[
\limsup_{n \to \infty} \frac{\E{\sup_{\fl{n(1-\epsilon)} \le i \le n} X^n(i)}}{\alpha(n)} \le 2\epsilon^{4/3} + \epsilon^{2/3} - \epsilon^{4/3}.
\]

Hence, for any $\alpha(n)$, 
\begin{equation} \label{eqn:Xbound}
\lim_{\epsilon \to 0} \limsup_{n \to \infty}\frac{\E{\sup_{\fl{n(1-\epsilon)} \le i \le n} X^n(i)}}{\alpha(n) \vee\sqrt{n}} = 0
\end{equation}
and another application of Markov's inequality gives
\[
\lim_{\epsilon \to 0} \limsup_{n \to \infty} \Prob{ \frac{1}{\sqrt{n}}\sup_{\fl{n(1-\epsilon)} \le i \le n} X^n(i)  > \delta} = 0.
\]

(d) Recall that $(M_X^n(k), k \ge 0)$ defined by 
\begin{equation} \label{eqn:martingale}
M_X^n(k) = X^n(k) +\frac{1}{2}V^n(k) - L^n(k) + \sum_{i=0}^{k-1} \frac{2X^n(i) - 2 \I{X^n(i) > 0}}{3U^n(i) + 4V^n(i) + X^n(i) - \I{X^n(i) > 0}}
\end{equation}
is a martingale.  It follows that
\[
\begin{aligned}
\E{X^n(n) - X^n(\fl{n(1-\epsilon)}}& = \E{L^n(n) - L^n(\fl{n(1 - \epsilon)})} + \frac{1}{2} \E{V^n(\fl{n(1-\epsilon)}) - V^n(n)} \\
& \qquad  - \E{\sum_{i=\fl{n(1-\epsilon)}}^{n-1}\frac{2X^n(i) - 2 \I{X^n(i) > 0}}{3U^n(i) + 4V^n(i) + X^n(i) -  \I{X^n(i) > 0}}}.
\end{aligned}
\]
Now, from (\ref{eqn:rootn}) we have
\[
\E{N^n(n) - N^n(\fl{n(1-\epsilon)})} \ge \E{\sum_{i=\fl{n(1-\epsilon)}}^{n-1} \frac{X^n(i) - \I{X^n(i) > 0}}{3U^n(i) + 4V^n(i) + X^n(i) -\I{X^n(i) > 0}}}.
\]
Hence,
\[
\E{L^n(n) - L^n(\fl{n(1 - \epsilon)})} \le \E{X^n(n)} + 2\E{N^n(n) - N^n(\fl{n(1-\epsilon)})}.
\]
Applying Markov's inequality and using (\ref{eqn:Nnbound}) and (\ref{eqn:Xbound}), we get that for any $\delta > 0$,
\[
\lim_{\epsilon \to 0} \limsup_{n \to \infty} \Prob{\frac{1}{\alpha(n) \vee \sqrt{n}} (L^n(n) - L^n(\fl{n(1 - \epsilon)})) > \delta} = 0. \qedhere
\]
\end{proof}

\begin{lem} \label{Lem Upper bound on number of nodes of degree 3 at eta n}
We have
\begin{equation*}
U^{n}(n)\leq{N^{n}(n)}+3V^{n}(0).
\end{equation*}
\end{lem}

\begin{proof}
At each step we either connect to a node of degree $3$ or $4$, or there is a clash. Each node of degree $3$ is removed after we connect to it, and each node of degree $4$ can be visited at most twice before being removed. Thus, during the first $n$ steps, a node of degree $3$ is removed at least $n-N^{n}(n)-2V^{n}(0)$ times. Moreover, at most $V^{n}(0)$ nodes of degree $4$ become nodes of degree $3$.  Hence,
\[
U^n(n) \le U^n(0) - (n - N^n(n) - 2V^n(0)) + V^n(0) = N^n(n) + 3V^n(0). \qedhere
\]
\end{proof}

We finally turn to the behaviour of our processes on the time-interval $[n,\zeta_n]$.

\begin{prop} \label{prop:ntozetan}
As $n \to \infty$, 
\begin{enumerate}
\item[(a)] $\zeta_n/n \convprob 1$,
\item[(b)] $\frac{1}{\omega(n) \vee \sqrt{n}} \sup_{n \le i \le \zeta_n} X^n(i) \convprob 0$,
\item[(c)] $\frac{1}{\omega(n) \vee \sqrt{n}} \left( N^n(\zeta_n) - N^n(n) \right) \convprob 0$,
\item[(d)] $\frac{1}{\omega(n) \vee \sqrt{n}} \left( L^n(\zeta_n) - L^n(n) \right) \convprob 0.$
\end{enumerate}
\end{prop}

\begin{proof} Write $\bar{\alpha}(n) = \alpha(n) \vee \sqrt{n}$.

(a) As in each step we remove at least two half-edges, by the same argument as gave us (\ref{eqn:crudebound}), we have $\zeta_n - n \le 3 U^n(n) + 4V^n(n) \le 3U^n(n) + 4\alpha(n)$ and so Lemma~\ref{Lem Upper bound on number of nodes of degree 3 at eta n} implies that 
\begin{align}\label{eq zeta minus n}
\zeta_n - n \le 3 N^n(n) + 13\alpha(n).
\end{align}
If $\alpha(n)/\sqrt{n} \to a$ then $N^n(n)/\sqrt{n} \convdist N_1^a$, and it follows that $\zeta_n/n \convprob 1$ and, indeed, that $(\zeta_n - n)/\sqrt{n}$ is a tight sequence of random variables.   If, on the other hand, $\alpha(n) > > \sqrt{n}$, we have $N^n(n)/\alpha(n) \convprob 1/4$, and again we get $\zeta_n/n \convprob 1$ with $(\zeta_n - n)/\alpha(n)$ a tight sequence of random variables.

(b) Consider $\sup_{n \le i \le \zeta_n} X^n(i)$.  By the usual coupling with $Y^n$, we have
\[
\sup_{n \le i \le \zeta_n} X^n(i) \le V^n(n) +  \sup_{n \le i \le \zeta_n} Y^n(i),
\]
where
\[
(Y^n(n+i))_{0 \le i \le \zeta_n-n} \equidist (2 + |X_1^n(n) + Z(i)|)_{0 \le i \le \zeta_n - n}
\]
Fix $\delta > 0$.  Then for any $C > 13$, since $\frac{\sup_{0 \le i \le C\bar{\alpha}(n)} |Z(i)|}{\sqrt{\bar{\alpha}(n)}}$ is bounded in $L^2$, we have
\begin{align*}
& \limsup_{n \to \infty} \Prob{\frac{\sup_{0 \le i \le \zeta_n-n} |Z(i)|}{\bar{\alpha}(n)} > \delta/2} \\
& \qquad \le \limsup_{n \to \infty} \Prob{\frac{\sup_{0 \le i \le C\bar{\alpha}(n)} |Z(i)|}{\bar{\alpha}(n)} > \delta/2} + \limsup_{n \to \infty} \Prob{\zeta_n - n > C \bar{\alpha}(n)} \\
& \qquad \le \Prob{N_1 > (C -13)/3},
\end{align*}
where for the last inequality we recall (\ref{eq zeta minus n}). We have $V^n(n)/\bar{\alpha}(n) \convprob 0$ by Proposition~\ref{prop:1-epsto1} (a) and $X^n(n)/\bar{\alpha}(n) \convprob 0$, so since $C > 13$ was arbitrary, we obtain
\[
\frac{1}{\bar{\alpha}(n)} \sup_{n \le i \le \zeta_n} X^n(i) \convprob 0.
\]

(c) We have
\begin{equation} \label{eqn:split}
\begin{aligned}
& \Prob{\frac{1}{\bar{\alpha}(n)} (N^n(\zeta_n) - N^n(n)) > \delta} \\
& \le \Prob{\frac{1}{\bar{\alpha}(n)}(N^n(\zeta_n) - N^n(n)) > \delta, \frac{N^n(n)}{\bar{\alpha}(n)} \le \frac{1}{\epsilon} } + \Prob{\frac{N^n(n)}{\bar{\alpha}(n)} > \frac{1}{\epsilon}} \\
& \le \frac{1}{\delta \bar{\alpha}(n)}\E{(N^n(\zeta_n) - N^n(n)) \I{\frac{N^n(n)}{\bar{\alpha}(n)} \le \frac{1}{\epsilon}} }  + \epsilon \frac{\E{N^n(n)}}{\bar{\alpha}(n)}.
\end{aligned}
\end{equation}

Let $\eta_n = \inf\{i \ge n: U^n(i) \le \epsilon \bar{\alpha}(n)\}$.  We have that  $\zeta_n - \eta_n \le 3U^n(\eta_n) + 4V^n(n) \le 3 \epsilon \bar{\alpha}(n)$ and $\zeta_n - n \le 3 N^n(n)$.  Now
\begin{equation} \label{eqn:finalbound}
\begin{aligned}
& \E{(N^n(\zeta_n) - N^n(n)) \I{\frac{N^n(n)}{\bar{\alpha}(n)} \le \frac{1}{\epsilon}}}  \\
& \le  \E{\sum_{i=n}^{\zeta_n - 1} \frac{3U^n(i) + 2(X^n(i) - \I{X^n(i) > 0})(3U^n(i) + 4V^n(i) + X^n(i) - 3)}{2(3U^n(i) + 4V^n(i) + X^n(i) -\I{X^n(i) > 0})(3U^n(i) + X^n(i) -\I{X^n(i) > 0} - 2)} \I{\frac{N^n(n)}{\bar{\alpha}(n)} \le \frac{1}{\epsilon}}}   \\
& \le \E{\sum_{i=n}^{\zeta_n - 1} \left( \frac{X^n(i) + 1}{3 U^n(i)} \wedge 1\right) \I{\frac{N^n(n)}{\bar{\alpha}(n)} \le \frac{1}{\epsilon}}}  \\
& \le \E{\sum_{i=n}^{\eta_n} \left(\frac{X^n(i) + 1}{3 \epsilon \bar{\alpha}(n)}\wedge 1\right) \I{\frac{N^n(n)}{\bar{\alpha}(n)} \le \frac{1}{\epsilon}}} + 3 \epsilon \bar{\alpha}(n)  \\
& \le \frac{1}{\epsilon^2} \E{ \left(\left(1+\sup_{n \le i \le \zeta_n} X^n(i) \right) \wedge 3 \epsilon \bar{\alpha}(n) \right) \I{\frac{N^n(n)}{\bar{\alpha}(n)} \le \frac{1}{\epsilon}}} + 3 \epsilon \bar{\alpha}(n). 
\end{aligned}
\end{equation}
By the coupling, we have that the expectation is bounded by
\[
\E{\left(3 + V^n(n) + X^n(n) + \sup_{0 \le i \le 3\bar{\alpha}(n)/\epsilon}  |Z(i)| \right) \wedge 3 \epsilon \bar{\alpha}(n)}.
\]
Now
\[
\frac{1}{\bar{\alpha}(n)} \left(3 + V^n(n) + X^n(n) + \sup_{0 \le i \le 3\bar{\alpha}(n)/\epsilon}  |Z(i)| \right) \wedge 3 \epsilon
\]
is a bounded random variable which converges to 0 in probability, and so from (\ref{eqn:finalbound}) we get
\[
\limsup_{n \to \infty} \frac{1}{\bar{\alpha}(n)} \E{(N^n(\zeta_n) - N^n(n)) \I{\frac{N^n(n)}{\bar{\alpha}(n)} \le \frac{1}{\epsilon}}}  \le 3 \epsilon.
\]
Now note that combining Theorem~\ref{thm fluid limit for X for omega>>sqrt n},  (\ref{eqn:NnUI}) and (\ref{eqn:Nnbound}), we have that there exists a constant $C > 0$ such that
\[
\sup_{n \ge 1} \frac{\E{N^n(n)}}{\bar{\alpha}(n)} < C
\]
and so using (\ref{eqn:split}) we obtain
\[
 \limsup_{n \to \infty} \Prob{\frac{1}{\bar{\alpha}(n)} (N^n(\zeta_n) - N^n(n)) > \delta} \le \frac{3 \epsilon}{\delta} +  C \epsilon.
\]
Since $\epsilon > 0$ was arbitrary, we get 
\[
 \limsup_{n \to \infty} \Prob{\frac{1}{\bar{\alpha}(n)} (N^n(\zeta_n) - N^n(n)) > \delta} = 0.
\]
Finally, we observe that using the martingale $M_X^n$ defined in (\ref{eqn:martingale}) and the optional stopping theorem, we have
\begin{align*}
0 & \le \E{(L^n(\zeta_n) - L^n(n)) \I{\frac{N^n(n)}{\bar{\alpha}(n)} \le \frac{1}{\epsilon}}} \\
& \le \E{(X^n(\zeta_n) - X^n(n))\I{\frac{N^n(n)}{\bar{\alpha}(n)} \le \frac{1}{\epsilon}}} + 2 \E{(N^n(\zeta_n) - N^n(n)) \I{\frac{N^n(n)}{\bar{\alpha}(n)} \le \frac{1}{\epsilon}}} \\
& \le 2 \E{(N^n(\zeta^n) - N^n(n))\I{\frac{N^n(n)}{\bar{\alpha}(n)} \le \frac{1}{\epsilon}}}
\end{align*}
since $X^n(\zeta_n) = 0$ and $X^n(n) \ge 0$.  So then
\begin{align*}
& \Prob{ \frac{1}{\bar{\alpha}(n)} \left( L^n(\zeta_n) - L^n(n) \right) > \delta}  \\
&  \qquad \le \Prob{\frac{1}{\bar{\alpha}(n)}(L^n(\zeta_n) - L^n(n)) > \delta, \frac{N^n(n)}{\bar{\alpha}(n)} \le \frac{1}{\epsilon} } + \Prob{\frac{N^n(n)}{\bar{\alpha}(n)} > \frac{1}{\epsilon}} \\
& \qquad \le \frac{2}{\delta \bar{\alpha}(n)}\E{(N^n(\zeta_n) - N^n(n)) \I{\frac{N^n(n)}{\bar{\alpha}(n)} \le \frac{1}{\epsilon}} }  + \epsilon \frac{\E{N^n(n)}}{\bar{\alpha}(n)}
\end{align*}
and the rest of the argument goes through as before.
\end{proof}

We now have all the elements needed to prove Theorem~\ref{thm:scalinglimit}.

\begin{proof}[Proof of Theorem~\ref{thm:scalinglimit}]
(i) From Theorem~\ref{Thm Convergence up to 1-epsilon} and Lemma~\ref{Lem X N joint convergence up to 1-epsilon}, for any $\epsilon > 0$ we have
\[
\left(\frac{X^n(\fl{ns})}{\sqrt{n}}, \frac{L^n(\fl{ns})}{\sqrt{n}}, \frac{N^n(\fl{ns})}{\sqrt{n}}, 0 \le s \le 1-\epsilon \right) \convdist \left(X_s, L_{s},N_s, 0 \le s \le 1-\epsilon \right).
\]
By Proposition~\ref{Prop sistem 1}, we have
\[
\left(X_{1-\epsilon}, L_{1-\epsilon}, N_{1-\epsilon} \right) \to \left(0, L_1, N_1 \right)
\]
almost surely as $\epsilon \to 0$.  The principle of accompanying laws (Theorem 3.2 of Billingsley~\cite{Billingsley}) gives that the statement of Proposition~\ref{prop:1-epsto1} is exactly what we need in order to deduce that
\[
\left(\frac{X^n(\fl{ns})}{\sqrt{n}}, \frac{L^n(\fl{ns})}{\sqrt{n}}, \frac{N^n(\fl{ns})}{\sqrt{n}}, 0 \le s \le 1 \right) \convdist \left(X_s, L_s, N_s, 0 \le s \le 1 \right).
\]
Proposition~\ref{prop:ntozetan} then allows us to conclude that
\[
\left(\frac{X^n(\fl{ns})}{\sqrt{n}}, \frac{L^n(\fl{ns})}{\sqrt{n}}, \frac{N^n(\fl{ns})}{\sqrt{n}}, 0 \le s \le \zeta_n/n \right) \convdist \left(X_s, L_s, N_s, 0 \le s \le 1 \right),
\]
as desired.

(ii) From Theorem~\ref{thm fluid limit for X for omega>>sqrt n}, for any $\epsilon > 0$, we have
\[
\left( \frac{X^n(\fl{ns})}{\alpha(n)}, \frac{L^n(\fl{ns})}{\alpha(n)}, \frac{N^n(\fl{ns})}{\alpha(n)}, 0 \le s \le 1-\epsilon \right) \convdist \left(x(s), 0, m(s), 0 \le s \le 1-\epsilon \right).
\]
It is straightforward that $x(1-\epsilon) \to 0$ and that $m(1-\epsilon) \to 1/4$ as $\epsilon \to 0$. By another application of the principle of accompanying laws, and Propositions~\ref{prop:1-epsto1} and \ref{prop:ntozetan}, we may then conclude that
\[
\left( \frac{X^n(\fl{ns})}{\alpha(n)}, \frac{L^n(\fl{ns})}{\alpha(n)}, \frac{N^n(\fl{ns})}{\alpha(n)}, 0 \le s \le \zeta_n/n \right) \convdist \left(x(s), 0, m(s), 0 \le s \le 1 \right),
\]
as desired.
\end{proof}

%
%
\section{Appendix} 
\subsection{Proof of the invariance principle}

\begin{proof}[Proof of Theorem \ref{Theorem Invariance principle}]
Let ${\theta}^n(r) = \inf\{t \ge 0: Q^n(t) > \int_0^t \sup_{|y| \le r} q(s,y) ds + 1 \}$.  Set $\tilde{M}^n_r := M^n( \cdot \wedge \theta^n(r) \wedge \tau^n(r))$.  Relative compactness of $\{\tilde{M}^n_r\}_{n \ge 1}$ follows as in Theorem 4.1 of Chapter 7 of Ethier and Kurtz.  Assumptions (c) and (d) imply that any subsequential weak limit has sample paths in $\mathbb{C}(\R_+,\R)$ almost surely.  Then all of the conditions in Assumption 4.1 of Kang and Williams \cite{KangWilliams} hold for the processes stopped at $\tau^n(r)$.  It follows by their Theorem 4.2 that the sequence of processes
\[
\{(Y^n,M^n,L^n)\}_{n \ge 1}
\]
is tight and such that any subsequential weak limit almost surely has continuous sample paths.

Now fix $r_0 > 0$ and let $\{Y^{n_k}(\cdot \wedge \tau^{n_k}(r_0)), M^{n_k}(\cdot \wedge \tau^{n_k}(r_0)), L^{n_k}(\cdot \wedge \tau^{n_k}(r_0)) \}_{k \ge 1}$ be a convergent subsequence with limit $Y_{r_0}, M_{r_0}, L_{r_0}$.  Let $\tau_{r_0}(r) = \inf\{t \ge 0: |Y_{r_0}(t)| \ge r\}$.  Then for all but countably many $r < r_0$ (i.e.\ those such that $\Prob{\lim_{s \to r} \tau_{r_0}(s) = \tau_{r_0}(r)} = 1$), we have
\begin{align*}
& (Y^{n_k}(\cdot \wedge \tau^{n_k}(r)), M^{n_k}(\cdot \wedge \tau^{n_k}(r)), L^{n_k}(\cdot \wedge \tau^{n_k}(r)), \tau^{n_k}(r)) \\
& \convdist (Y_{r_0}(\cdot \wedge \tau_{r_0}(r)), M_{r_0}(\cdot \wedge \tau_{r_0}(r)), L_{r_0}(\cdot \wedge \tau_{r_0}(r)), \tau_{r_0}(r)).
\end{align*}
Conditions (f), (c) and (d) guarantee that $M^{n_{k}}\left({\cdot\wedge{\tau^{n_{k}}(r_{0})}}\right)$ is a uniformly integrable martingale. Conditions (b), (g) and (e) guarantee that $(M^{n_{k}}\left({\cdot\wedge{\tau^{n_{k}}(r_{0})}}\right))^2 - Q^{n_k}\left({\cdot\wedge{\tau^{n_{k}}(r_{0})}}\right)$ is a uniformly integrable martingale. Thus, by Problem 7 from Chapter 7 of Ethier and Kurtz, as in the proof of Theorem 1.4(b) there, we must then have that
\[
M_{r_0}(t \wedge \tau_{r_0}(r)) = Y_{r_0}(t \wedge \tau_{r_0}(r)) - L_{r_0}(t \wedge \tau_{r_0}(r)) - \int_0^{t \wedge \tau_{r_0}(r)} b(s, Y_{r_0}(s)) ds
\]
and
\[
M_{r_0}(t \wedge \tau_{r_0}(r))^2 - \int_0^{t \wedge \tau_{r_0}(r)} q(s, Y_{r_0}(s))ds
\]
are continuous martingales.  It follows that
\[
M_{r_0}(t \wedge \tau_{r_0}(r)) = \int_0^{t \wedge \tau_{r_0}(r)} \sigma(s, Y_{r_0}(s)) dW(s)
\]
for some Brownian motion $W$.  But then we have
\[
Y_{r_0}(t \wedge \tau_{r_0}(r)) = \int_0^{t \wedge \tau_{r_0}(r)} b(s,  Y_{r_0}(s)) ds + \int_0^{t \wedge \tau_{r_0}(r)} \sigma(s, Y_{r_0}(s)) dW(s) + L_{r_0}(t \wedge \tau_{r_0}(r)).
\]
Moreover, by (i), $L_{r_0}(0) = 0$, $L_{r_0}$ is non-decreasing, and $\int_0^{t \wedge \tau_{r_0}(r)} \I{Y_{r_0}(s) > 0} d
L_{r_0}(s) = 0$.  So $(Y_{r_0}, L_{r_0})$ solves a stopped version of the reflected SDE.

But if $(Y,L)$ is the unique solution of the reflected SDE with Brownian motion $W$ then $(Y(\cdot \wedge \tau(r)), L(\cdot \wedge \tau(r)))$ is the unique solution of the stopped version, where $\tau(r) = \inf\{t \ge 0: |Y(s)| \ge r\}$.  In consequence, $(Y^n(\cdot \wedge \tau^n(r)), L^n(\cdot, \wedge \tau^n(r))) \convdist (Y(\cdot \wedge \tau(r)), L(\cdot \wedge \tau(r))$ for any $r$ such that $\Prob{\lim_{s \to r}\tau(s) = \tau(r)} = 1$.  But $\tau(r) \to \infty$ as $r \to \infty$, since $Y$ has continuous sample paths.  Hence, $(Y^n, L^n) \convdist (Y,L)$.
\end{proof}

\subsection{A hitting probability calculation}
Fix $d \in (0,1/2)$ and an integer $b \ge 1$.  Suppose that $A$ is a random walk with step-distribution
\begin{align*}
\Prob{A(i+1) - A(i) = -1 |A(i)} & = \frac{1}{2}, \\
\Prob{A(i+1) - A(i) = 1 |A(i)} & = \frac{1}{2} - d, \\
\Prob{A(i+1) - A(i) = 2 |A(i)} & = d.
\end{align*}

\begin{lem} \label{lem:hittingprob}
The probability, started from 1, that $A$ hits $0$ before $\{b,b+1\}$ is given by
\[
1 - \tfrac{4d (-1-\sqrt{1+8d})^{b} + 4d (-1+\sqrt{1+8d})^b }{(-2)^{b+1}\sqrt{1+8d} + 4d(-1-\sqrt{1+8d})^b + 4d(-1+\sqrt{1+8d})^b - (-1-\sqrt{1+8d})^{b+1} - (-1+\sqrt{1+8d})^{b+1}}.
\]
\end{lem}

\begin{proof}
For $0 \le k \le b+1$, let
\[
h_k = \Prob{\text{$A$ hits $0$ before $\{b,b+1\}$} | A(0) = k}.
\]
Then for $1 \le k \le b-1$,
\[
h_k = \frac{1}{2} h_{k-1} + \left(\frac{1}{2} - d\right) h_{k+1} + d h_{k+2}
\]
and elementary calculations yield that
\[
h_k = \varphi + \beta \lambda_{-}^k + \gamma \lambda_+^k, \quad 0 \le k \le b+1,
\]
where
\[
\lambda_{\pm} = \frac{-1 \pm \sqrt{1+8d}}{4d}.
\]
The boundary conditions are $h_0 = 1$, $h_b = h_{b+1} = 0$.  Solving for the constants gives
\[
\begin{aligned}
 \varphi & = \frac{\lambda_{-}^{b+1} \lambda_{+}^b  - \lambda_{-}^b \lambda_{+}^{b+1}}{\lambda_{-}^{b+1} \lambda_+^b - \lambda_{-}^b \lambda_{+}^{b+1} - \lambda_{-}^{b+1} + \lambda_{-}^b - \lambda_+^{b+1} + \lambda_{+}^b}, \\
 \beta & = \frac{\lambda_{+}^b - \lambda_{+}^{b+1}}{\lambda_{-}^{b+1} \lambda_+^b - \lambda_{-}^b \lambda_{+}^{b+1} - \lambda_{-}^{b+1} + \lambda_{-}^b - \lambda_+^{b+1} + \lambda_{+}^b}, \\
\gamma & = \frac{\lambda_{-}^b - \lambda_{-}^{b+1}}{\lambda_{-}^{b+1} \lambda_+^b - \lambda_{-}^b \lambda_{+}^{b+1} - \lambda_{-}^{b+1} + \lambda_{-}^b - \lambda_+^{b+1} + \lambda_{+}^b}.
\end{aligned}
\]
Simplifying the $k=1$ case, we obtain
\[
h_1 = 1 - \tfrac{4d (-1-\sqrt{1+8d})^{b} + 4d (-1+\sqrt{1+8d})^b }{(-2)^{b+1}\sqrt{1+8d} + 4d(-1-\sqrt{1+8d})^b + 4d(-1+\sqrt{1+8d})^b - (-1-\sqrt{1+8d})^{b+1} - (-1+\sqrt{1+8d})^{b+1}}. \qedhere
\]
\end{proof}

\section*{Acknowledgments}
We are very grateful to James Martin for valuable discussions. We would like to thank Ruth Williams for directing us to her paper \cite{KangWilliams} which very much facilitated our proof of the invariance principle. C.G.'s research is supported by EPSRC Fellowship EP/N004833/1. 

\bibliographystyle{abbrv}

\end{document}